\documentclass[a4paper, leqno,11pt]{amsart}

\usepackage{amssymb, amsmath, amsthm}
\usepackage{xargs} 
\usepackage{hyperref} 
\usepackage[colorinlistoftodos]{todonotes} 
\usepackage{subcaption,graphicx,xcolor}
\usepackage{forest}

\numberwithin{equation}{section}
\newtheorem{thm}{Theorem}

\newtheorem{corollary}[thm]{Corollary}
\newtheorem{lemma}[thm]{Lemma}
\newtheorem{proposition}[thm]{Proposition}
\newtheorem{conjecture}[thm]{Conjecture}

\newtheorem{qst}[thm]{Question}

\theoremstyle{remark}
\newtheorem{ex}[thm]{Example}
\newtheorem*{theorem*}{Theorem}

\newtheorem*{rmks}{Remarks}

\newtheorem{definition}{Definition}
 
\newcommand{\T}{\mathbb{T}}
\newcommand{\norm}[1]{\left\Vert#1\right\Vert}
\newcommand{\brkt}[1]{\left(#1\right)}

\newcommand{\abs}[1]{\left|#1\right|}
\newcommand{\set}[1]{\left\{#1\right\}}
\newcommand{\ip}[1]{\left\langle#1\right\rangle} 

	\newcommand{\N}{\mathbb{N}}

         \newcommand{\B}{\mathbb{B}}
	\newcommand{\C}{\mathbb{C}}
	\newcommand{\bS}{\mathbb{S}}
	\newcommand{\D}{\mathbb{D}}

	\newcommand{\cc}{\overline}

	\newcommand{\cd}{\mathfrak{D}} 
	\newcommand{\da}{\cd_\alpha}
	\newcommand{\norma}[1]{\norm{#1}_{\alpha}}
         \newcommand{\kerl}{\mathfrak{K}_{\lambda}}        
         \newcommand{\Pn}{\mathcal{P}_n} 
         \newcommand{\spn}{p_n^*}
 
	\newcommandx{\concern}[2][1=]{\todo[color = red!70!,#1]{Concern: #2}} 
	\newcommandx{\refq}[2][1=]{\todo[color = yellow!40!,#1,]{Reference: #2}} 
	\newcommandx{\wording}[2][1=]{\todo[color = violet!50!,#1,]{Wording: #2}} 
	\newcommandx{\alan}[2][1=]{\todo[color = green!25!,#1]{#2}}
	\newcommandx{\mere}[2][1=]{\todo[color = blue!25!,#1]{#2}}

        \newcommand{\McC}{\raise.5ex\hbox{c}}

\title[Optimal approximants in several variables]{Optimal approximants and orthogonal polynomials in several variables}

\author[Sargent]{Meredith Sargent}
\address{Department of Mathematics, University of Arkansas, Fayetteville, AR 72701, U.S.A.}
\email{sargent@uark.edu}

\author[Sola]{Alan A. Sola}
\address{Department of Mathematics, Stockholm University, 106 91 Stockholm, Sweden}
\email{sola@math.su.se}
\date{\today}

\thanks{AS acknowledges support from Ivar Bendixons stipendiefond f\"or docenter.}

\begin{document}

\begin{abstract} 
	We discuss the notion of optimal polynomial approximants in multivariable reproducing kernel Hilbert spaces. In particular, we analyze difficulties that arise in the multivariable case which are not present in one variable, for example, a more complicated relationship between optimal approximants and orthogonal polynomials in weighted spaces. Weakly inner functions, whose optimal approximants are all constant, provide extreme cases where nontrivial orthogonal polynomials cannot be recovered from the optimal approximants. Concrete examples are presented to illustrate the general theory and are used to disprove certain natural conjectures regarding zeros of optimal approximants in several variables.  
 \end{abstract}
\maketitle


%

\section{Introduction}
	There are situations when understanding a space $\mathcal{H}$ consisting of analytic functions on some open subset of $\C^d$ requires the analysis of functions of the form $1/f$, where $f\in \mathcal{H}$. In general, of course, $1/f\notin \mathcal{H}$ and it becomes natural to look for 
	substitutes $p^*\in \mathcal{H}$ that approximate $1/f$ in some appropriate sense. One example of this type of investigation is the problem of determining cyclic vectors for the shift operator, or shift operators, in a Hilbert function space. In this context, $f\in\mathcal{H}$ is cyclic if the polynomial multiples of $f$ form a dense subset of $\mathcal{H}$. If the constant function $1$ is assumed to be cyclic, then it is frequently the case that $f \in\mathcal{H}$ is cyclic if a constant function can be approximated in the norm of $\mathcal{H}$ by polynomial multiples of $f$. This in turn amounts, at least intuitively, to being able to approximate $1/f$ by polynomials.

	The notion of an optimal approximant to $1/f$ appeared some time ago, both in the mathematical literature and previously in the engineering literature under the name {\it least squares polynomial inverse} (typically in the setting of the Hardy space $H^2$). Chui \cite{Chui80} attributes the notion to E.A. Robinson who apparently considered such approximation problems in the context of stationary stochastic processes \cite{Rob63}. In the 80s, Chui and others \cite{Izu85} obtained several important results for one-variable $H^2$-approximants, in particular examining the location of their zeros. Least squares polynomial inverses were also studied systematically in the several complex variables setting by Delsarte, Genin, and Kamp in the late 70's. They were led to examine least squares polynomial inverses to functions in $H^2$ of the bidisk by problems in filtering theory \cite{Shanks72}.

	In a series of recent papers by several authors, cyclic vectors in Dirichlet-type spaces have been studied via polynomial substitutes to $1/f$, appearing there under the name {\it optimal approximants}; at the time, the authors were not aware of the earlier works mentioned above. Optimal approximants were initially considered, and in some cases computed explicitly, in the one-variable setting \cite{JAM15}, and were then used in \cite{PJM15} to exhibit non-cyclic polynomials in two-variable Dirichlet spaces. Subsequently, optimal approximants themselves have been studied in several papers, with a particular emphasis on the location of their zeros \cite{JLMS16,JentZeros19} and their boundary behavior and universality \cite{Betalprep1,Betalprep2}. See \cite{SecoSurvey} for a survey of optimal approximants.

	This present paper on optimal approximants has two complementary goals. One the one hand, we would like to draw the attention of the function theory and operator theory communities to some results and problems discussed in the engineering literature that, in our opinion, have not received enough attention. In some cases, we are also able to give simplified arguments and examples. On the other hand, we contribute to the theory in several ways. First, we explain how to extend the notion of optimal approximants to a more general several-variables setting: in principle, this part is straight-forward, but there are some technical points and choices that we need to pay particular attention to. We then show that many the nice finer properties exhibited by one-variable optimal approximants and related functions are lost in higher dimensions. Despite this, in some cases, particularly when examining orthogonal polynomials, we find a structure connected to the one variable case. Finally, natural conjectures for several variable-optimal approximants are disproved by examining specific examples.

	Our paper is structured as follows. We begin, in Section \ref{sRKHS}, by setting down notation and giving a brief overview of the function spaces we are interested in. We then define optimal approximants in reproducing kernel Hilbert spaces defined in domains in $\C^d$ and discuss ways of computing such approximants for a given target function. We also mention applications to the analysis of cyclic vectors and two-dimensional filters. In Section \ref{sInner}, we discuss weakly inner functions, which are singled out by their property of having constant optimal approximants, and their connections with classical inner functions. An idea from earlier papers in the one-variable setting is adapted to give an	explicit construction of weakly inner functions. In Section \ref{sOG} we examine how optimal approximants relate to orthogonal polynomials in weighted spaces, and investigate under what circumstances  orthogonal polynomials can be recovered from optimal approximants. We also show that for a certain class of examples, orthogonal polynomials in two variables can be found from the known one-variable case. Section \ref{sShanks} is devoted to zero sets of optimal approximants and, in particular, to what is known in the engineering literature as the Shanks conjecture on regions where optimal approximants are zero-free. We review some of the existing results and then present several counterexamples to possible Shanks-type statements. Finally, Section \ref{sComp} features explicit computation and discussion of optimal approximants and orthogonal polynomials for functions of the form $f=1-a(z_1+z_2)$. Throughout the paper, we attempt to give references to relevant previous work: we hope these sources will inspire further work even though it is likely we have overlooked some important contributions.

\section{Optimal Approximants in reproducing kernel Hilbert spaces}\label{sRKHS} 
	\subsection{Reproducing kernel Hilbert spaces}\label{ss:RKHS}
		Let $\Omega \subset \C^d$ be an open set containing the origin. A {\it reproducing kernel Hilbert space} $\mathcal{H}(\Omega)$ is a Hilbert space consisting of holomorphic functions on $\Omega$  such that evaluation at a point of $\Omega$,
		\[e_\lambda \colon \mathcal{H} \to \C, \quad e_{\lambda}[f]=f(\lambda),\]
		furnishes a bounded linear functional. By standard Hilbert space theory, there exists an element $\kerl \in \mathcal{H}$ with the reproducing property
		\[f(\lambda)=\langle f , \kerl \rangle_{\mathcal{H}},\]
		where $\langle \cdot,\cdot \rangle_{\mathcal{H}}$ denotes the inner product in $\mathcal{H}(\Omega)$. We call $\kerl$ the {\it reproducing kernel} at $\lambda$. For any orthonormal basis
		$\{\phi_j\}_{j=0}^{\infty}$ for $\mathcal{H}$, the reproducing kernel admits the series representation
		\[\kerl(z)=\sum_{j=0}^{\infty}\overline{\phi}_j(\lambda)\phi_j(z).\]
		See \cite{AMBook} for a general introduction to Hilbert function spaces.

		In this paper, we shall typically take $\Omega$ to be the unit disk, the unit bidisk, or the unit ball in $\C^d$. We shall also impose the standing assumptions that $\C[z_1,\ldots,z_d]$, the ring of polynomials in $d$ complex variables, forms a dense subspace of $\mathcal{H}(\Omega)$ and that 
		the operators of multiplication by the coordinate functions,
		\[S_j\colon \mathcal{H}\to \mathcal{H}, \quad S_j[f](z)=z_j\cdot f(z), \quad j=1,\ldots, d,\]
		act boundedly on $\mathcal{H}$.

		Throughout, we will consider the following spaces of holomorphic functions to illustrate the general theory.

		\subsection*{Dirichlet-type spaces in the disk and the bidisk}
			Let $\alpha \in (-\infty,\infty)$ be fixed. The {\it Dirichlet-type space} $D_{\alpha}$ consists of holomorphic functions $f=\sum_{k=0}^\infty a_k z^k$ on the unit disk $\D=\{z\in\C\colon|z|<1\}$ satisfying the norm boundedness condition
				\begin{equation}\label{Dnorm1}
					\norm{f}_{D_\alpha}^2 = \sum_{k=0}^\infty (k+1)^\alpha \abs{a_k}^2<\infty.		
			 	\end{equation} 
			When $\alpha=0$, we recover the standard Hardy space $H^2$. The choice $\alpha=-1$ corresponds to the Bergman space $A^2$ in the unit disk, while $D=D_1$ can be identified with the classical Dirichlet space consisting of functions having $\int_{\D}|f'(z)|^2dA(z)<\infty$, where $dA$ is normalized area measure on the disk. The literature on these spaces is vast but basic introductions can be found in \cite{DurBook, HedBook, ElFBook}.

			For $\alpha\in \{-1,0,1\}$, explicit expressions for the reproducing kernels are known. In $H^2$ and $A^2$, we have the usual {\it Szeg\H{o}} and {\it Bergman kernels}
			\begin{equation}
			\kerl^{H^2}(z)=\frac{1}{1-\overline{\lambda}z} \quad \textrm{and}\quad \kerl^{A^2}(z)=\frac{1}{(1-\overline{\lambda}z)^2}.
			\label{harberkernel}
			\end{equation}
			For non-integer values of $\alpha$, closed form expressions for the reproducing kernels $\kerl$ in terms of rational functions are in general not available.

			We can define Dirichlet-type spaces $\cd_{\alpha_1,\alpha_2}$ on the bidisk 
			\[\D^2=\set{(z_1,z_2)\in\C^2 \colon \abs{z_1}<1, \abs{z_2}<1}\]
			as tensor products of one-variable Dirichlet-type spaces; that is, we can take
			\[\cd_{\alpha_1,\alpha_2}=D_{\alpha_1}\otimes D_{\alpha_2}.\]
			See \cite{AMBook} for more on this perspective.
			In concrete terms, $\cd_{\alpha_1,\alpha_2}$ consists of holomorphic functions
			\[f(z_1,z_2)=\sum_{j=0}^\infty \sum_{k=0}^\infty a_{j,k}z_1^jz_2^k\] 
			on the bidisk whose Taylor coefficients satisfy
				\begin{equation}\label{Dnorm2}
					\norm{f}_{\alpha_1,\alpha_2}^2 = \sum_{j=0}^\infty \sum_{k=0}^\infty (j+1)^{\alpha_1}(k+1)^{\alpha_2}\abs{a_{j,k}}^2.
				\end{equation}
			We write $\da$ when $\alpha=\alpha_1=\alpha_2$, and the norm in this case will be denoted by $\norma{f}$. By the general theory of reproducing kernel spaces \cite{AMBook}, the kernel of $\cd_{\alpha_1,\alpha_2}$ at $\lambda=(\lambda_1,\lambda_2) \in \D^2$ is a product of one-variable kernels,
			\begin{equation}
			\mathfrak{K}^{\cd_{\alpha_1,\alpha_2}}_{\lambda_1,\lambda_2}(z_1,z_2)=\mathfrak{K}^{D_{\alpha_1}}_{\lambda_1}(z_1)\cdot \mathfrak{K}^{D_{\alpha_2}}_{\lambda_2}(z_2),\quad (z_1,z_2)\in \D^2.
			\label{kernelasproduct}
			\end{equation}

			Similar statements are valid in $d$-dimensional polydisks.

		\subsection*{The Drury-Arveson space}
			Let 
			\[\B^d=\{z\in \C^d\colon \|z\|^2<1\}\]
			denote the unit ball in $\C^d$ and let $\bS^d=\partial \B^d$ be its boundary, the unit sphere, and let
			\[\langle z,w\rangle=z_1\overline{w}_1+\cdots +z_d\overline{w}_d\]
			denote the standard Euclidean inner product on $\C^d$.

			The {\it Drury-Arveson space} on $\B^d$ is the reproducing kernel Hilbert space of holomorphic function of the ball determined by the kernel
			\[\kerl^{H^2_{d}}(z)=\frac{1}{1-\langle z, \lambda\rangle}, \quad z=(z_1,\ldots,z_d) \in \B^d.\]
			Basic structural properties of the Drury-Arveson space are discussed in, for instance, \cite{ShaBook,RicSun16}; for instance, the Drury-Arveson norm is invariant under unitaries. For our purposes, it will be useful to note that the norm in $H^2_d$ can be expressed in terms of the coefficients of $f=\sum_ka_kz^k$ using standard multi-index notation:
			\[\|f\|_{H^2_d}^2=\sum_{n=0}^{\infty}\sum_{|k|=n}\frac{k!}{|k|!}|a_k|^2.\]

			In particular, in two variables we have
			\[\kerl^{H^2_{2}}(z)=\frac{1}{1-\overline{\lambda}_1z_1-\overline{\lambda_2}z_2} \quad \textrm{and} \quad
			\|f\|_{H^2_2}^2=\sum_{j=0}^{\infty}\sum_{k=0}^{\infty}\frac{j!k!}{(j+k)!}|a_{j,k}|^2.\]

	\subsection{Optimal polynomial approximants}\label{ssOA_def}
		Set $\chi_0=1$ and let 
		\[\chi_1,\chi_2, \chi_3, \ldots\] 
		be an ordering of complex monomials $z^k=z_1^{k_1}\cdots z_d^{k_d}$ according to some chosen order. In several variables there are several natural ways to index monomials: the general setup below is independent of this choice, but when we later turn to examples we typically use the degree lexicographic order \cite{GWsiam07} where monomials are ordered by increasing total degree, and ties are broken lexicographically. In two variables, this amounts to
		\[\chi_1=z_1, \quad \chi_2=z_2,  \quad \chi_3=z_1^2, \quad \chi_4=z_1z_2, \quad \chi_5=z_2^2, \quad \chi_6=z_1^3,\]
		and so on. We find it illuminating to display the monomials in the tree diagram in Figure \ref{fig:triangle}.
			{\renewcommand{\arraystretch}{1.25}
			\begin{figure}[h!]
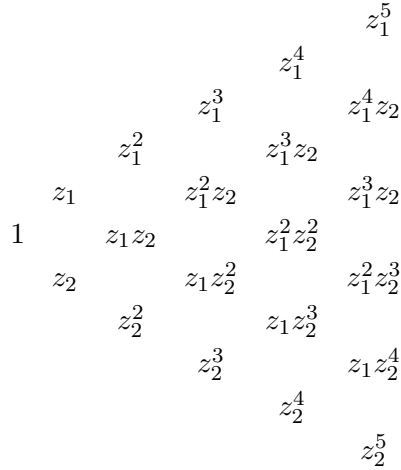

				\centering
					\begin{equation*}
						\begin{matrix}
							~		&		~		&		~			&		~				&		~ 				&		~z_1^5			\\
							~		&		~		&		~			&		~				&		z_1^4			&		~	 					\\
							~		&		~		&		~			&	z_1^3			&		~ 				&		z_1^4z_2		\\
							~		&		~		&	z_1^2		&		~				&		z_1^3z_2 	&		~						\\	
							~		&	 z_1 	&		~			& z_1^2z_2 	&		~					&		z_1^3z_2		\\
							1		&		~		&	z_1z_2 	&		~				&		z_1^2z_2^2&		~						\\
							~		&	 z_2 	&		~			& z_1z_2^2 	&		~					&		z_1^2z_2^3	\\
							~		&		~		&	z_2^2		&		~				&		z_1z_2^3	&		~						\\	
							~		&		~		&		~			&	z_2^3			&		~ 				&		z_1z_2^4		\\
							~		&		~		&		~			&		~				&		z_2^4			&		~ 					\\
							~		&		~		&		~			&		~				&		~					&		z_2^5				\\
						\end{matrix}
					\end{equation*}
				\caption{\small We shall call the center row of this diagram the \emph{main diagonal.} Degree lexicographic order reads ``down'' each column (where the total degree of the monomials in each column is fixed), moving to the right.\label{fig:triangle}}
			\end{figure}}


		With an ordering of monomials in place, we set
		\begin{equation}
		\Pn=\mathrm{span}\{\chi_j\;\colon j=0,\ldots, n\},\quad n=0,1,2,\dots.
		\label{Pndef}
		\end{equation}
		Since $\chi_0=1$ we have $\mathcal{P}_0=\mathrm{span}\{1\}$, the constant polynomials.
		Note that \[\mathcal{P}_0\subset \mathcal{P}_1\subset \cdots\subset \Pn\subset \cdots \] is
		 an exhaustion of $\C[z_1,\ldots, z_d]$ (viewed as a vector space) by finite-dimensional subspaces, that is,  $\bigcup_n\Pn=\C[z_1,\ldots, z_d]$. If $d=1$, we typically order monomial by degree, in which case
		\[\Pn=\{p \in \C[z]\colon \deg(p)\leq n\}\] for $n=0,1,2,\ldots.$
		\begin{definition}
		Let $f\in \mathcal{H}(\Omega)$ be given. The {\it $n$th-order optimal polynomial approximant to $1/f$} with respect to $\Pn$ is defined as
		\[\spn(z)=\mathrm{Proj}_{f\cdot \Pn}[1](z),\]
		where $\mathrm{Proj}_{f\cdot \Pn}\colon\mathcal{H}\to f\cdot \Pn$ denotes the orthogonal projection onto the subspace $f\cdot \Pn$.

		In other words, $\spn$ is the unique polynomial that minimizes $\|p\cdot f-1\|_{\mathcal{H}}$ among all $p\in \Pn$.
		\end{definition}
		The existence and uniqueness of $\spn$, relative to a particular choice of $\{\chi_j\}$, follows immediately from Hilbert space theory: our assumption that multiplication by each variable acts boundedly on $\mathcal{H}$ implies that $f\cdot \Pn$ is a closed subspace of $\mathcal{H}$ for each $n$.	Note that we obtain different sequences of optimal approximants depending on the contents of the $\Pn$.

		The following notion of distance will also feature.
		\begin{definition}
		For a given $f\in \mathcal{H}$ and $\Pn$ as above, the {\it $n$th order optimal norm} is defined as
		\[\nu_n(f, \mathcal{H})=\|p_n^*\cdot f-1\|_{\mathcal{H}}.\]

		\end{definition}
		Note that, since the subspaces $\Pn$ are nested, $\nu_n(f)$ is non-increasing as a function of $n$.

	\subsection{Optimal approximants via Grammians}\label{ssGrammian}
		The following is a straightforward reinterpretation of previous methods of computing optimal approximants \cite{JAM15,FMS14} to our present setting.
		\begin{proposition}\label{matrixprop}
			Let $f\in \mathcal{H}\setminus \{0\}$. Then the coefficients of the $n$-order optimal approximant $\spn=\sum_{j=0}^nc_j^*\chi_j$ are given by solution to the linear system
			\[M\vec{c}\,^\ast=\vec{b},\]
			where $M$ is an $(n+1)\times (n+1)$ Grammian matrix with entries given by
			\[M_{ij}=\langle\chi_jf, \chi_i f\rangle\]
			and 
			\[\vec{b}=\left(\begin{array}{c}\langle 1, \chi_0f\rangle \\\vdots\\ \langle1,\chi_n f \rangle\end{array}\right).\] 
		\end{proposition}
		The proof is analogous to the one-variable case, see \cite{FMS14}; we sketch it for the reader's convenience.
		\begin{proof}
			By the definition of $\spn$, we have $(\spn f-1)\perp \Pn$. Thus $\langle \spn f-1, f\chi_i\rangle=\langle1, f\chi_i \rangle$. This in turn can be rewritten as
			\[\left\langle\sum_jc^*_j \chi_j f,f\chi_i \right\rangle=\langle 1, f\chi_i\rangle,\]
			and, using linearity, we obtain the desired linear system.
		\end{proof}
		If $\langle 1, \chi_if\rangle=0$ for all $i\geq 1$, as is the case for most spaces we are interested in, then $\vec{b}=\overline{f(0)}\delta_{i,0}$ in the above proposition.

		It is typically not straightforward to find $\{p_n^*\}$ in closed form for a given $f$ using the representation in Proposition \ref{matrixprop}. More sophisticated approaches to computation and fine analysis of optimal approximants are discussed in \cite{DGKalg79, Betalprep2}, for example, but are not needed for what we want to achieve in this paper.

		Building on one-variable work in \cite{JAM15,FMS14}, we can obtain optimal approximants for some simple polynomial targets.
		\begin{ex}\label{FMSexpoly}
			Consider a sequence $\{\omega(k)\}_{k=0}^{\infty}$ of strictly positive weights satisfying $\lim_{k\to \infty}\omega(k+1)/\omega(k)=1$ and let $\mathcal{H}_{\omega}$ be the Hilbert function space consisting of analytic $f\colon \D\to \C$ whose power series $f=\sum_{k=0}^{\infty}a_kz^k$ satisfy
			\begin{equation}
			\|f\|_{H_{\omega}}=\sum_{k=0}^{\infty}\omega(k)|a_k|^2<\infty.
			\end{equation}
			Let us further assume that $\{z^k/\|z^k\|_{\omega}\}$ is an orthonormal basis for $\mathcal{H}_{\omega}$. In this setting, Fricain, Mashreghi, and Seco \cite[Theorem 3.9]{FMS14} have found an explicit expression for the $\mathcal{H}_{\omega}$-optimal approximants to $1/f$ for the function $f=1-z$. (See \cite{JAM15} for the case of Dirichlet-type spaces in the unit disk.)  Indeed, we have
			\[\spn(z)=\sum_{k=0}^{n}\left(1-\frac{\sum_{j=0}^k\omega(j)^{-1}}{\sum_{j=0}^{n+1}\omega(j)^{-1}}\right)z^k.\]
		\end{ex}

		In our discussion of higher-dimensional analogs of Example \ref{FMSexpoly}, we find it convenient to consider diagonal subspaces
		\[\mathcal{J}_{n}=\mathrm{span}\{(z_1\cdots z_d)^k\colon k=0,1,\ldots, n\} \quad \textrm{and}\quad 
		\mathcal{J}=\mathrm{span}\{(z_1\cdots z_d)^k\colon k\in \N\}.\]
		Also, define $\Pn$ using degree lexicographic order and let $\oslash n$ denote the lowest index $m$ for which the exponent $(z_1z_2)^n$ belongs to $\mathcal{P}_m$. (Explicitly, $\oslash 1=4$, $\oslash 2=12$, and so on, and note that $\mathcal{J}_n \subsetneq \mathcal{P}_{\oslash n}$.)

		\begin{ex}\label{diriexpoly}
			We first consider optimal approximants to $1/(1-z_1z_2)$ in the Dirichlet-type spaces $\mathfrak{D}_{\alpha_1,\alpha_2}$ in the bidisk.

			In \cite{PJM15,JAM19}, it was observed that there is an isometric isomorphism between $\mathcal{J}_n$, 
			viewed as a closed subspace of $\mathfrak{D}_{\alpha_1,\alpha_2}$, and the set $J_{n}=\mathrm{span}\{z^k\colon k=0,1, \ldots, n\}$ viewed as a closed subspace of $D_{\alpha_1+\alpha_2}$, 
			a Dirichlet-type space in the unit disk. Under this isomorphism, $f=1-z_1z_2$ is mapped to $F=1-z$. Next, we note that $D_{\alpha_1+\alpha_2}$ can be viewed as $\mathcal{H}_{\omega}$ with weight sequence $\omega(k)=(k+1)^{\alpha_1+\alpha_2}$. Finally, by orthogonality, the $n$th-order optimal approximants $\spn=\mathrm{Proj}_{f\cdot \Pn}[1]$ are polynomials in $z_1z_2$ only. 

			Thus the optimal approximants $\spn$ change from $\mathcal{J}_m$ to $\mathcal{J}_{m+1}$, and stay the same for all $\Pn$ containing $\mathcal{J}_m$ and being strictly contained in $\mathcal{P}_{\oslash (m+1)}$.

			Now, using Example \ref{FMSexpoly}, we  find that  
			\begin{equation}
			p^*_{\oslash n}=\mathrm{Proj}_{F\cdot P_n}[1](z_1z_2)=\sum_{k=0}^{n}\left(1-\frac{\sum_{j=0}^k(j+1)^{-(\alpha_1+\alpha_2)}}{\sum_{j=0}^{n+1}(j+1)^{-(\alpha_1+\alpha_2)}}\right)(z_1z_2)^k
			\end{equation}
			are the optimal approximants to $1/(1-z_1z_2)$ for $\oslash n\leq k< \oslash (n+1)$.
		\end{ex}

		\begin{ex}\label{DAexpoly}
		A similar analysis applies in the case of the $d$-dimensional Drury-Arveson space (or Dirichlet-type spaces in the unit ball, cf. \cite{S15}). 

		By the arithmetic-geometric means inequality, the mapping 
		\[\iota: (z_1,\ldots, z_d)\mapsto d^{d/2}\prod_{j=1}^d z_j\]
		sends the unit ball $\B^d$ to the unit disk $\D$.
		Next, we note that
		\[\|(z_1\cdots z_d)^k\|^{2}_{H^2_d}=\frac{(k!)^d}{(dk)!}.\]
		Together, these observations establish an isometric isomorphism between $\mathcal{J}_k$ viewed as a closed subspace of $H^2_d$ and the set $J_n$ sitting inside
		the space $\mathcal{H}_{\omega}$ of functions on the disk associated with the weight sequence
		\[\omega_d(k)=d^{dk}\frac{(k!)^d}{(dk)}.\]
		Using this choice of weight sequence in the formula in Example \ref{FMSexpoly}, we obtain the polynomials 
		\[p^*_{\bf{\oslash} n}(z_1,\ldots, z_d)=\sum_{k=0}^{n}\left(1-\frac{\sum_{j=0}^k\omega_d(j)^{-1}}{\sum_{j=0}^{n+1}\omega_d(j)^{-1}}\right)(d^{d/2}z_1\cdots z_d)^k,\]
		and these are the optimal approximants to $1/(1-d^{d/2}\prod_{k=1}^dz_k)$ in $H^2_d$ for $\mathbf{\oslash} n\leq k<\mathbf{\oslash} (n+1)$. Here, $\oslash$ is the $d$-dimensional analog of $\oslash$ in two variables.

		For instance, in the two variable case, 
		\[p^*_{\bf{\oslash} 1}(z_1,z_2)=\frac{1}{3}, \qquad 
				p^*_{\bf{\oslash} 2}(z_1,z_2)=\frac{7}{15}+\frac{2}{15}z_1z_2,\]
				\[p^*_{\bf{\oslash} 3}(z_1,z_2)=\frac{19}{35} + \frac{22}{35} z_{1} z_{2} + \frac{4}{7} z_{1}^{2} z_{2}^{2},\]
			and so on.
		\end{ex}
		
	\subsection{Applications of optimal approximants: cyclic vectors}\label{ssCyclicity}
		Recall that a vector $f\in \mathcal{H}$ is said to be {\it cyclic} for the shift operators $S_1, \ldots, S_d$ if the invariant subspace
		\[[f]_{\mathcal{H}}=\mathrm{clos}_{\mathcal{H}}\mathrm{span}\left\{S_1^{k_1}\cdots S_d^{k_d}f\colon k\in \N^d\right\}\]	
		is dense in $\mathcal{H}$. Since the polynomials were assumed dense in all the Hilbert spaces we are considering, the function $f=1$ is a cyclic vector.	As is explained in 
		\cite{BScyc84,JAM15}, this is equivalent to having
		\[\nu_n(f)\to 0 \quad \textrm{as}\quad n\to \infty.\]
		One of the original applications of optimal approximants in \cite{JAM15,PJM15} was to use the rate at which $\nu_n(f)$ decays to zero to not only distinguish between cyclic and non-cyclic vectors, but also to give finer distinctions between ``how cyclic" different cyclic functions are.

		\begin{ex}\label{bidiskcyclicex}
			Returning to the one-variable spaces $H_{\omega}$ considered by Fricain, Mashreghi, and Seco, and the function $f=1-z$, we note that, by \cite[Corollary 3.10]{FMS14},
			\begin{equation}
			\nu_n(1-z, \mathcal{H}_{\omega})=\left(\frac{1}{\sum_{k=0}^{\infty}\frac{1}{\omega(k)}}\right)^{1/2}.
			\label{FMSdistance}
			\end{equation}

			Arguing as in Example \ref{FMSexpoly}, we can now use \eqref{FMSdistance} to extract information about cyclicity of $f=1-z_1z_2$ in the Dirichlet spaces $\mathfrak{D}_{\alpha_1,\alpha_2}$ in the bidisk, and about $1-d^{d/2}\prod_{k=1}^dz_k$ in the Drury-Arveson space. 

			Since the weight sequence in $\mathfrak{D}_{\alpha_1, \alpha_2}$ is $\omega_1(k)\omega_2(l)=(k+1)^{\alpha_1}(l+1)^{\alpha_2}$, we have $\omega(k)=\omega_1(k)\omega_2(k)=(k+1)^{\alpha_1+\alpha_2}$, and thus
				\[\nu_{\oslash n}(1-z_1z_2, \mathfrak{D}_{\alpha_1,\alpha_2})=\nu_n(1-z, D_{\alpha_1+\alpha_2})=\left(\frac{1}{\sum_{k=0}^{n+1}(k+1)^{-(\alpha_1+\alpha_2)}}\right)^{1/2}.\]
			The sum in the right-hand side converges as $n$ tends to infinity precisely when $\alpha_1+\alpha_2\leq 1$. Thus, as was shown in \cite{JAM19}, $f=1-z_1z_2$ is cyclic in $\mathfrak{D}_{\alpha_1,\alpha_2}$ if and only if $\alpha_1+\alpha_2\leq 1$. When $\alpha_1+\alpha_2>1$, we obtain
				\[\nu^2(1-z_1z_2, \mathfrak{D}_{\alpha_1,\alpha_2})=\frac{1}{\sum_{k=0}^{\infty}(k+1)^{-(\alpha_1+\alpha_2)}}=\frac{1}{\zeta(\alpha_1+\alpha_2)}.\]
			In particular, for the Dirichlet space $\mathfrak{D}=\mathfrak{D}_{1,1}$, $\nu(1-z_1z_2, \mathfrak{D})=\frac{\sqrt{6}}{\pi}$.
		\end{ex}

		Cyclic polynomials for $\mathfrak{D}_{\alpha_1,\alpha_2}$ have been completely characterized, see \cite{NGN70,TAMS16,JAM19}, and the cyclicity/non-cyclicity part of Example \ref{bidiskcyclicex} follows immediately from that characterization. What optimal approximants allow us to do, is to measure how far from cyclic $1-z_1z_2$ is for different pairs of 
		$(\alpha_1,\alpha_2)$.

		\begin{ex}\label{ex:DAdiag}
			We turn to the Drury-Arveson space $H^2_d$ and the functions $f=1-d^{d/2}z_1\cdots z_d$. As in Example \ref{FMSexpoly},
			we set $\omega_d(k)=d^{dk}\frac{(k!)^d}{(dk)!}$ and obtain
			\[\nu_{\oslash n}^2\left(1-d^{d/2}\prod_{k=1}^dz_k, H^2_d\right)=\frac{1}{\sum_{k=0}^{n+1}\omega_d(k)^{-1}}\]
			as well as
			\[\nu^2\left(1-d^{d/2}\prod_{k=1}^dz_k, H^2_d\right)=\frac{1}{\sum_{k=0}^{\infty}\omega_d(k)^{-1}}.\]
			A short computation involving Stirling's formula shows that
			\[\omega_d(k)\asymp k^{\frac{d-1}{2}},\quad \textrm{as}\quad k\to \infty.\]
			In particular, $f=1-d^{d/2}\prod_{k=1}^dz_k$ is cyclic in $H^2_d$ if and only if $d\leq 3$. This recovers an earlier result of Richter and Sundberg \cite{RicSun12} who used the same embedding argument above, which also features in Arveson's work \cite{arvy}.
		\end{ex}
		
	\subsection{Applications of optimal approximants: two-dimensional recursive filters}\label{filtersub}
		Another, older, application of optimal approximants relates to two-dimensional recursive filtering theory and was discussed by Shanks, Treitel, and Justice \cite{Shanks72}. We give a brief description of their work here, and note that their work in turn was motivated by engineering applications including the study of seismic records and photographic data \cite{Shanks72}.

		Given a data array $D=(d_{j,k})_{j,k=1}^n$, we form a two-variable polynomial $D(z_1,z_2)=\sum_{j=1}^n\sum_{k=1}^nd_{j,k}z_1^{j-1}z_2^{k-1}$. Then, in the notation of \cite{Shanks72}, a recursive filter algorithm is obtained as follows. We set
		\begin{equation}
		R(z_1,z_2)=F(z_1,z_2)D(z_1,z_2)
		\label{recfiltereq}
		\end{equation}
		where $F(z_1,z_2)=A(z_1,z_2)/B(z_1,z_2)$ is a rational function of two variables. After clearing fractions, \eqref{recfiltereq} translates into
		\[B(z_1,z_2)R(z_1,z_2)=A(z_1,z_2)D(z_1,z_2).\]
		Assuming that the constant term $b_{1,1}$ in $B(z_1,z_2)=\sum_{j=1}^{M_B}\sum_{k=1}^{N_B}b_{j,k}z_1^{j-1}z_2^{k-1}$ is non-zero and dividing through, we obtain
		\begin{multline*}
		R(z_1,z_2)=\left(\sum_{j=1}^{N_A}\sum_{k=1}^{N_A}\frac{a_{j,k}}{b_{1,1}}z_1^{j-1}z_2^{k-1}\right)D(z_1,z_2)\\
		-\left(\sum_{\parbox{1.6cm}{\centering{\tiny $1\leq j\leq M_{B}$,\\ $1\leq k\leq N_{B}$,\\ $(j,k)\neq \vec{1}$}}}
						\frac{b_{j,k}}{b_{1,1}}z_1^{j-1}z_2^{k-1}\right)R(z_1,z_2).
		\end{multline*}
		We thus have
		\[r_{m,n}=\displaystyle\sum_{j=1}^{M_A}\sum_{k=1}^{N_A}\frac{a_{j,k}}{b_{1,1}}d_{m-j+1, n-k+1}\quad
			-\sum_{\parbox{1.6cm}{\centering{\tiny $1\leq j\leq M_{B}$,\\ $1\leq k\leq N_{B}$,\\ $(j,k)\neq \vec{1}$}}}
						\frac{b_{j,k}}{b_{1,1}}r_{m-j+1,n-k+1},\]
		expressing the output coefficient $r_{m,n}$ in terms of output coefficients which are either assumed to have been previously computed or are set to zero.

		In order for this scheme to be of practical use, it is desirable that the filter is {\it stable}, that is, that bounded inputs $D$ are transformed into bounded outputs $R$. In light of \eqref{recfiltereq}, one expects that this would require that $B(z_1,z_2)\neq 0$ for some subset of values $(z_1,z_2)$. Indeed, Justice and Shanks proved \cite{JS73} that stability holds if and only if $B(z_1,z_2)\neq 0$ on $\overline{\D^2}$. Unfortunately, this need not hold for all potentially useful filters $F=A/B$.

		To get around this difficulty, Shanks, Treitel, and Justice proposed replacing the two-variable function $B$ by its $H^2$-optimal approximants $\spn$. They argued that, intuitively speaking, $\spn$ should retain ``many" of the features of $B$. Moreover, in light of the one-variable case and numerical evidence in two variables, they conjectured \cite{Shanks72} that two-variable optimal approximants should be non-vanishing in the closed bidisk. Thus $1/\spn$ would be a stabilizing filtering substitute for $1/B$.

		Unfortunately, the Shanks-Treitel-Justice approach to stabilization does not work without additional assumptions on the target function $B$ since there are polynomials $B$ whose optimal approximants $\spn$ vanish inside the bidisk, making the filter $1/\spn$ unstable as well.

		We discuss zero set problems for optimal approximants  in Section \ref{sShanks}.
	
\section{Optimal approximants and weakly inner functions}\label{sInner}
	\subsection{Weakly Inner Functions}\label{ssWInner}
		Certain functions in $\mathcal{H}(\Omega)$ have the distinguishing property that their optimal approximants do not change as we increase 
		$\Pn$. Following \cite{DGKasym80, CMB18}, we make the following definition.
		\begin{definition}
		We say that $g\in\mathcal{H}(\Omega)\setminus\{0\}$ is {\it weakly inner} if
		\[\langle g, \chi_j g\rangle=0 \quad \textrm{for all} \quad j\neq 0.\]
		\end{definition}
		See \cite{ChMaRo19} for a comprehensive overview of notions of innerness for a wide range of reproducing kernel Hilbert spaces. Inner functions can also be defined for Banach spaces of analytic functions in a similar fashion using the notion of {\it Birkhoff-James orthogonality}, viz. \cite[Section 7]{ChMaRo19}.

		\begin{proposition}
		If $g\in \mathcal{H}(\Omega)$ is weakly inner, then its optimal approximants are all equal to a single constant: $\spn=p_0$ for $n=1,2,\ldots.$
		\end{proposition}
		\begin{proof}
			By Proposition \ref{matrixprop}, the coefficients of $\spn=\sum_{j=0}^nc_j^*\chi_j$ are given by
			\[M\vec{c}\,^*=\overline{g}(0)\delta_{k,0},\]
			where $M_{j,k}=\langle g\chi_j,g\chi_k\rangle$ and $\delta_{k,0}=(1,0,\ldots, 0)^T$. By assumption, $M_{j,0}=c \delta_{k,0}$, so the first column of $M$ consists of all zeros past the first entry which is $c=\|g\|^2$. By elementary linear algebra, the inverse matrix $M^{-1}$ has the same property. But then
			\[\vec{c}\,^*=M^{-1}\bar{\theta}(0)\delta_{k,0}=\frac{\bar{\theta}(0)}{c}\delta_{k,0},\]
			and the proof is complete.
		\end{proof}
		\begin{corollary}
		If $g\in \mathcal{H}$ is weakly inner then $\nu_n(g)=\nu_0(g)$ for all $n=1, 2\ldots$.
		\end{corollary}
		
		One obvious class of weakly inner functions in the Hardy space is the class of {\it classical inner functions}: recall that a bounded holomorphic function 
		$\theta \colon \D^d \to \C$ is said to be {\it inner} if $|\theta(\zeta)|=1$ for almost every $\zeta \in \T^d$.

		\begin{lemma}\label{lemWIconstOA}
			Suppose $\theta\colon \D^d\to \C$ is inner. Then $\theta$ is weakly $H^2$-inner. 
		\end{lemma}
		\begin{proof}
			Without loss of generality, we may assume $\theta(0,0)\neq 0$. Since $\theta$ is inner, we have 
			\begin{equation*}
				\langle \chi_j \theta, \chi_k \theta \rangle=\int_{\T^d}\chi_j \chi_k |\theta|^2dm=\int_{\T^d}\chi_j \chi_k dm=0,\quad \text{if } j\neq k.
			\end{equation*}
			Thus the matrix $M$ is diagonal, and
				$M^{-1}\bar{\theta}(0,0) e_1=\bar{\theta}(0,0)\delta_{k,0}$, as claimed.
		\end{proof}				
				
		When $d=1$, inner functions and weakly $H^2$-inner functions coincide. Weakly inner functions in the Bergman space of the unit disk are precisely the Bergman-inner functions \cite[Chapter 3]{HedBook}. In higher dimensions, however, a new phenomenon manifests itself and the class of classically inner functions forms a subclass of all weakly $H^2$-inner functions. This was originally observed by Delsarte, Genin, and Kamp, see \cite[Section 8]{DGKasym80}, who gave a power series example. In the next subsection, we give simpler examples.
				
	\subsection{Shapiro-Shields functions}\label{sssShapiroShields}
		By adapting a construction in \cite{CMB18}, which in turn is based on an older idea of H.S. Shapiro and A.L. Shields \cite{ShaShi62}, we can build weakly inner functions in any reproducing kernel Hilbert space with a finite prescribed zero set. See \cite{ChMaRo19} and \cite{Le19} for further generalizations.
		\begin{definition}
			Let $\Lambda=\{\lambda_1,\ldots,\lambda_n\} \in \Omega\setminus\{0\}$ be a given set of distinct points and let $\kerl^{\mathcal{H}}$ be the reproducing kernel of $\mathcal{H}$ at a point $\lambda$. Define $\mathfrak{K}_{\Lambda}$ to be the $n\times n$ matrix whose entries are given by $(\mathfrak{K}_{\Lambda})_{i,j}=\langle \mathfrak{K}_{\lambda_i}, \mathfrak{K}_{\lambda_j}\rangle$ and let $\vec{1}=(1,1,\ldots,1)\in \C^n$.

			The \emph{Shapiro-Shields function} for $\mathcal{H}$ associated with $\Lambda$ is defined as 
			\begin{equation} \label{eqShaShi}
				s_{\Lambda}(z)=
				\begin{vmatrix}
					1 &  \vec{1}\\
			 		(\mathfrak{K}_{\lambda_j})_{j=1}^n   &\mathfrak{K}_{\Lambda}.
				\end{vmatrix}
				\end{equation}
			
		\end{definition}
		
		The normalized Shapiro-Shields function is defined as
			\begin{equation}
			g_{\Lambda}(z)=\frac{s_{\Lambda}(z)}{\|s\|_{\Lambda}};
			\end{equation}
			normalization is not essential for our purposes.

		\begin{proposition}\label{shapshiprop}
			Suppose $\mathcal{H}(\Omega)$ is a reproducing kernel Hilbert space with monomials $\{\chi_j\}$ forming an orthogonal set.		
			Let $s_{\Lambda}$ be a Shapiro-Shields function for $\mathcal{H}$ associated with a finite set $\Lambda \in \Omega \setminus \{0\}$ of distinct points. Then $s_{\Lambda}$ is weakly inner in $\mathcal{H}(\Omega)$, and $s_{\Lambda}$ vanishes at each point of $\Lambda$.
		\end{proposition}
		\begin{proof}
			The proof is a straight-forward adaptation of the one-variable proof in \cite{CMB18} and is sketched for the reader's convenience. First, the fact that $s_{\Lambda}$ vanishes at each $\lambda_j$ follows from the fact that the first column in the determinant defining $s_{\Lambda}$ is equal to the $j+1$ column. 
			 
			To see that $s_{\Lambda}$ is non-trivial, it suffices to note that the kernels $\mathfrak{K}_{\lambda_1}, \ldots, \mathfrak{K}_{\lambda_n}$ are linearly independent. 

			To establish that $s_{\Lambda}$ is weakly inner, we perform a cofactor expansion of the second argument of  
				$\langle \chi_j s_{\Lambda}, s_{\lambda}\rangle$ 
			along the first column, 
			\begin{equation*}
				\langle \chi_j s_{\Lambda}, s_{\Lambda}\rangle
				=
				\cc{\det \mathfrak{K}_{\Lambda}}\langle \chi_j s_{\Lambda}, 1\rangle+\sum_{m=1}^{n}A_m\langle\chi_js_{\Lambda}, \mathfrak{K}_{\lambda_m}\rangle.
			\end{equation*}

			Finally, $\langle \chi_js_{\Lambda}, 1\rangle=0$ for $j\geq 1$ by orthogonality of monomials, while each $\langle \chi_j s_{\Lambda}, \mathfrak{K}_{\lambda_m}\rangle=\chi_j(\lambda_m)s_{\Lambda}(\lambda_m)$ is zero since $s_{\Lambda}(\lambda_m)=0$ for $m=1,\ldots, n$.
 		\end{proof}	
		
		For Hardy and Bergman spaces in the unit disk, normalized Shapiro-Shields functions recover well-known inner functions, see \cite{CMB18}. Here, we examine such functions in the bidisk and the ball.

		\begin{ex}
	 		The Shapiro-Shields function for $H^2(\D^2)$ associated with a point $(\lambda_1,\lambda_2)\in \D^2$ is
			\begin{equation*}
			 s_{\lambda}(z)=\frac{1}{(1-|\lambda_1|^2)(1-|\lambda_2|^2)}\frac{\overline{\lambda}_1(\lambda_1-z_1)+\overline{\lambda}_2(z_2-\lambda_2)-\overline{\lambda_1\lambda_2}(\lambda_1\lambda_2-z_1z_2)}{(1-\overline{\lambda}_1z_1)(1-\overline{\lambda}_2z_2)}.
			 \end{equation*}
			Several remarks are in order. As in one variable, the rational function $s_{\lambda}$ extends holomorphically to a bigger polydisk, whose radius depends on $\lambda$. Next, since $s_{\lambda}$ above is holomorphic of two variables, the function vanishes at points of the bidisk other than $\lambda$. If $\lambda_1=0$ or $\lambda_2=0$, we recover a multiple of a one-variable Blaschke factor, but in general $s_{\lambda}$ is not of product type. 

		  Finally,  since $s_{\lambda}$ violates the Rudin-Stout description of rational inner functions in polydisks \cite[Chapter 5]{RudSBook}, $s_{\lambda}$ is not inner in the classical sense.
		\end{ex}
		
		\begin{ex}
			In the Bergman space $A^2(\D^2)$, the Shapiro-Shields function associated with $(\lambda_1,\lambda_2)\in \D^2$ is
			\begin{multline*} 
				s_{\lambda}(z)= \left(\frac{1}{(1-|\lambda_1|^2)^2(1-|\lambda_2|^2)^2(1-\overline{\lambda}_1z_1)^2(1-\overline{\lambda}_2z_2)^2}\right)\cdot\\
					\quad
						\Big(
							\left(\cc{\,\lambda_1\lambda_2}\right)^2\left(z_1^2z_2^2-\lambda_1^2\lambda_2^2\right)
							+2\cc{\lambda_1}^2\,\cc{\lambda_2}\left(\lambda_1^2\lambda_2-z_1^2z_2\right)
							+2\cc{\lambda_1}\,\cc{\lambda_2}^2\left(\lambda_1\lambda_2^2-z_1,z_2^2\right)\\
							\qquad
							+\cc{\lambda_1}^2\left(z_1^2-\lambda_1^2\right)
							+4\cc{\,\lambda_1\lambda_2}\left(z_1z_2-\lambda_1\lambda_2\right)
							+\cc{\lambda_2}^2\left(z_2^2-\lambda_2^2\right)\\
							\qquad
							+2\cc{\lambda_1}\left(\lambda_1-z_1\right)
							+2\cc{\lambda_2}\left(\lambda_2-z_2\right)
						\Big)
			\end{multline*}			
		\end{ex}

		\begin{ex}
			For $d\geq 1$, let $\lambda \in \B^d$ be a point in the unit ball. The Shapiro-Shields function for $H^2_d$ associated with $\lambda$ is
			\[s_{\lambda}(z)=\frac{1}{1-\|\lambda\|^2}\frac{\langle \lambda-z, \lambda \rangle}{1-\langle z, \lambda\rangle}.\]
		\end{ex}
It would be interesting to conduct a systematic study of weakly inner functions in general reproducing kernel Hilbert spaces.
\section{Orthogonal polynomials}\label{sOG}
	\subsection{Optimal approximants and orthogonal polynomials}\label{ssOAsAndOGs}
		Another interesting aspect of optimal approximants is their connection to orthogonal polynomials of certain weighted spaces. This is discussed in one variable in \cite{JLMS16}, Section 3, where the authors write the optimal approximants in terms of orthogonal polynomials and exploit properties of orthogonal polynomials to show that, in the case of the Hardy space, optimal approximants are zero free in the unit disk. These connections were also observed by engineers in, e.g., \cite{GKstab77}, who also showed that they extend to the two variable Hardy space case and form the basis for the Shanks conjecture about the location of zeros of optimal approximants (Section \ref{sZerosShanks}).

		This relationship can be generalized to a reproducing kernel Hilbert space $\mathcal{H}(\Omega),$ with properties discussed in Section \ref{ss:RKHS} and inner product $\langle \cdot\,, \cdot \rangle_{\mathcal{H}}$. Recall that for $f\in \mathcal{H}(\Omega)$, the $n$th-order optimal polynomial approximant to $1/f$ with respect to $\Pn$ is defined as \(\spn(z)=\mathrm{Proj}_{f\cdot \Pn}[1](z)\). If we let $\set{f\phi_j}$ be an orthonormal basis for $f\cdot \Pn$, then we can consider the $\phi_j$ to be orthonormal polynomials in a weighted space $\mathcal{H}_f$ with inner product
		\begin{equation} \label{eq:weightednorm}
			\langle g, h\rangle_{\mathcal{H}_f} :=\langle gf , hf\rangle_{\mathcal{H}}.
		\end{equation}
		To avoid trivialities, we assume that $f$ is not identically zero, and does not vanish at the origin. Using the orthonormal basis $\set{f\phi_j}$, $f\spn$ can be expanded as
		\begin{equation}
			\left(f\spn\right)(z_1,z_2)=\sum_{k=0}^{n} \ip{1,f\phi_k}_{\mathcal{H}}\phi_k(z_1,z_2)f(z_1,z_2),
		\end{equation}
		and we can cancel to get
		\begin{equation}
			\spn(z_1,z_2)=\sum_{k=0}^{n} \ip{1,f\phi_k}_{\mathcal{H}}\phi_k(z_1,z_2)
		\end{equation}
		This in turn implies that
		\begin{equation}
			\langle1, f\phi_k\rangle \phi_n(z_1,z_2)=\spn(z_1,z_2)-p_{n-1}^*(z_1,z_2), \quad n=1,2,3\ldots.
			\label{OAvsOGformula}
		\end{equation}
		In certain favorable circumstances, the relation \eqref{OAvsOGformula} allows us to recover orthogonal polynomials from optimal approximants. When $f$ is weakly inner, however, the optimal approximants $\spn$ are all equal to the same constant. Therefore, the orthogonal polynomials cannot be extracted from the formula \eqref{OAvsOGformula}. In fact, we have the following.

		\begin{lemma} \label{OGiplemma}
			Suppose that for some $f\in \mathcal{H}(\Omega)\setminus\{0\}$ and some $n\in \N$, we have $\spn=p^*_{n-1}$. Then $\langle 1,f\phi_n\rangle=0$ for non-constant $\phi$. 
		\end{lemma}

		The main example considered in the engineering applications is the weighted Hardy space of the bidisk with inner product given by
		\begin{equation*}
			\langle g,h\rangle_{H^2,f} = \lim_{r\to 1^-}\int_0^{2\pi}\int_{0}^{2\pi} g(re^{i\theta_1},re^{i\theta_2}) \cc{h(re^{i\theta_1},re^{i\theta_2})} \abs{f(re^{i\theta_1},re^{i\theta_2})}^2 d\theta_1 d\theta_2
		\end{equation*}
		where $d\theta_1$ and $d\theta_2$ are normalized Lebesgue measure on the circle. Similarly, for the Bergman space in the bidisk we have
		\[\langle g,h\rangle_{A^2,f}=\iint_{\mathbb{D}^2}g(z_1,z_2)\overline{h(z_1,z_2)}|f(z)|^2dA(z_1)dA(z_2)\]
but for general pairs $(\alpha_1,\alpha_2)$, the inner product $\langle\cdot, \cdot\rangle_f$ is not expressible as a weighted integral of $g$ and $h$ over the bidisk.		
		 In all the $\da$ spaces, however,
			\[\ip{1,f\phi_k}=\cc{f(0)\phi_k(0)},\]
		and we obtain the following immediate consequence of Lemma \ref{OGiplemma}.

		\begin{lemma}
		Suppose that for some $f\in\da \setminus\{0\}$ and some $n\in \N$, we have $\spn=p^*_{n-1}$. Then $\phi_n(0)=0$.

		In particular, if $f$ is weakly inner in $\da$, then all orthogonal polynomials $\phi_n$ in $\mathfrak{D}_{\alpha_1,\alpha_2,f}$ vanish at the origin for $n\geq 1$.
		\end{lemma}

		\begin{ex}
		If $f$ is a classical inner function in $H^2(\D^d)$ then the weighted norm $\langle, \cdot, \cdot \rangle_{f}$ coincides with the usual $H^2$ norm, and the set of monomials $\{z_1^kz_2^l\}_{k,l\in \N}$ yields orthogonal polynomials in the weighted space, all vanishing at the origin whenever $(k,l)\neq(0,0)$. 

		For weakly inner but not classically inner functions, one expects orthogonal polynomials to exhibit a more complicated structure.
		\end{ex}

	\subsection{A class of weighted orthogonal polynomials}
		Our simple example $f(z_1,z_2)=1-az_1z_2$ ($a=1$ in the bidisk $a=\sqrt{2}$ in the $2$-ball) exhibits optimal approximants and orthogonal polynomials with interesting behavior. As discussed in Examples \ref{diriexpoly} and \ref{DAexpoly}, the optimal approximants to $1/f$ contain only monomials of the form $(z_1z_2)^n$, that is, monomials on the main diagonal in Figure \ref{fig:triangle}. Because of this, not all of the orthogonal polynomials for the weight $f$ can be reconstructed from the optimal approximants for $1/f$. This is similar to the case of a weakly inner function, but, in contrast, the differences of the optimal approximants do give non-constant polynomials that are orthogonal, just not all the polynomials needed to span the $\mathcal{P}_n$. For instance, the polynomial $\chi_1=z_1$ cannot be expressed as a linear combination of polynomials in $z_1z_2$.

		Here, we assume the reproducing kernel Hilbert space $\mathcal{H}(\Omega)$ discussed in Section \ref{ss:RKHS} has the additional property  that the monomials are pairwise orthogonal. (The Drury-Arveson space and each Dirichlet-type space have this property.) Exploiting the diagonal structure in the monomial ordering allows us express the full collection of orthogonal polynomials for a reproducing kernel Hilbert space weighted by a general polynomial in $z_1z_2$ (as in \eqref{eq:weightednorm}) in terms of one-variable polynomials. We begin with a lemma about the inner products of monomials in the weighted space $\mathcal{H}_f$.

		\begin{lemma}\label{lem:ogstar}
			Let $f(z_1,z_2)=1+a_1z_1z_2 + a_2(z_1z_2)^2 +\cdots+a_N(z_1z_2)^N$ be a polynomial and let $\mathcal{H}$ be a reproducing kernel Hilbert space in which the monomials are orthogonal. Consider $\mathcal{H}_f$, the space weighted by $f$ with inner product $\langle g, h\rangle_{\mathcal{H}_f} :=\langle gf , hf\rangle_{\mathcal{H}}$. For nonnegative integers $\ell_1\leq k_1$, $\ell_2\leq k_2$, and for an integer $J$ such that $0\leq J \leq \min(k_1,k_2,N)$,
				\begin{equation*}
					\ip{z_1^{k_1}z_2^{k_2},\,z_1^{\ell_1}z_2^{\ell_2}}_f=
						\begin{cases}
							\displaystyle\sum_{n=0}^{N-J} a_n\cc{a_{n+J}} 
									\norm{z_1^{k_1+n}z_2^{k_2+n}}_f^2 &\quad \mathrm{if}\;\;\parbox{2.5cm}{$\ell_1=k_1-J$\\ $\ell_2=k_2-J$},\\
									~&~\\
							0 &\quad \mathrm{otherwise.}
						\end{cases}
				\end{equation*}
		\end{lemma}
		\begin{proof}
			Expanding the inner product gives
			\begin{align}
				\ip{z_1^{k_1}z_2^{k_2},\,z_1^{\ell_1}z_2^{\ell_2}}_f
				&=\ip{z_1^{k_1}z_2^{k_2}f,\,z_1^{\ell_1}z_2^{\ell_2}f}_{\mathcal{H}}\nonumber\\
				&=\sum_{m=0}^N\sum_{n=0}^N a_n\cc{a_m} \ip{z_1^{k_1+n}z_2^{k_2+n},\,z_1^{\ell_1+m}z_2^{\ell_2+m}}_{\mathcal{H}}.\label{eq:ipsum_before}
			\end{align}
			Because the monomials are orthogonal in $\mathcal{H}$,
			\begin{equation*}\label{eq:ipzero}
				\ip{z_1^{k_1+n}z_2^{k_2+n},\,z_1^{\ell_1+m}z_2^{\ell_2+m}}_{\mathcal{H}}=
				\begin{cases}
					\norm{z_1^{k_1+n}z_2^{k_2+n}}_{\mathcal{H}}^2 &\quad \mathrm{if}\;\;\parbox{2.5cm}{$\ell_1+m=k_1+n$\\ $\ell_2+m=k_2+n$},\\
					~&~\\
					0 &\qquad \mathrm{otherwise.}
				\end{cases}
			\end{equation*}
			Since $0\leq \ell_1 \leq k_1$ and $0\leq \ell_2 \leq k_2$, there are integers $J_1$, $J_2$ such that $0\leq J_1 \leq k_1$ and $0\leq J_2 \leq k_2$ with $\ell_1=k_1-J_1$ and $\ell_2=k_2-J_2$. For each term of the sum, $m-n$ is fixed, so for the nonzero terms, where \(\ell_1+m=k_1+n\) and \(\ell_2+m=k_2+n\),
				\[J_1=k_1-\ell_1=m-n=k_2-\ell_2=J_2,\]
			so let $J=J_1=J_2$. 

			Then, the conditions for the inner product to be non zero,
			$$\ell_1+m=k_1+n \quad \mathrm{and}\quad \ell_2+m=k_2+n,$$ 
			become $m=n+J,$ and because $0\leq m,n \leq N$, $\abs{m-n}\leq N$, so $J\leq N$. Finally, we can rewrite \eqref{eq:ipsum_before} as
			\begin{align*}
				\ip{z_1^{k_1}z_2^{k_2},\,z_1^{\ell_1}z_2^{\ell_2}}_f
				&=\sum_{m=0}^N\sum_{n=0}^N a_n\cc{a_m} \ip{z_1^{k_1+n}z_2^{k_2+n},\,z_1^{k_1-J+m}z_2^{k_2-J+m}}_{\mathcal{H}}\\
				&=\sum_{n=0}^{N-J} a_n\cc{a_{n+J}} \norm{z_1^{k_1+n}z_2^{k_2+n}}_f^2,
			\end{align*}
			when $\ell_1=k_1-J \;\mathrm{and}\; \ell_2=k_2-J.$ If these conditions do not hold, every term in the sum \eqref{eq:ipsum_before} will be zero.
		\end{proof}

		We now give a structural description of the full family of orthogonal polynomials for weights of the form $f=\sum_{k=0}^na_k(z_1z_2)^k$. Here, we shall consider the monomials in degree lexicographic order: 
			$$\chi_0=1,\, \chi_1=z_1,\, \chi_2=z_2,\, \chi_3=z_1^2,\quad \dots,$$
		and polynomial subspaces $\mathcal{P}_n=\mathrm{span}\set{\chi_0,\dots,\chi_n}.$ We let $\deg_{z_j}p$ denote the $z_j$-degree of a multi-variable polynomial. We consider orthogonal polynomials $\{\varphi_k\}$ for $\mathcal{H}_f$, ordered so that $\mathrm{span}\set{\varphi_0,\dots,\varphi_n}=\mathcal{P}_n$, and we assume that $\deg_{z_1}\varphi_k=\deg_{z_1}\chi_k$ and $\deg_{z_1}\varphi_k=\deg_{z_2}\varphi_k=\deg_{z_2}\chi_k$, and that each $\varphi_k$ is monic.
		\begin{thm}\label{thm:ogstar}
			For each $N\in \N_0$, let $$M=\max\set{\deg_{z_1}\varphi_N,\,\deg_{z_2}\varphi_N}-\min\set{\deg_{z_1}\varphi_N,\,\deg_{z_2}\varphi_N}.$$ There exists a unique $r_N \in \C[x]$ such that 
			\begin{enumerate}
				\item If $\deg_{z_1}\varphi_N  \geq  \deg_{z_2}\varphi_N$, then $\varphi_N=z_1^M r_N(z_1z_2)$
				\item If $\deg_{z_1}\varphi_N  \leq  \deg_{z_2}\varphi_N$, then $\varphi_N=z_2^M r_N(z_1z_2)$.
			\end{enumerate}
		\end{thm}
		The bidegree of each $r_N$ is implicit from degree lexicographical ordering.
		\begin{proof}
			Without loss of generality, assume $\chi_N=z_1^Az_2^B$ where $A\geq B$ so that $M=A-B.$ When $N=0,$ $\chi_0=1$, so $M=0,$ and $r_0(x)=1.$
			
			We proceed by induction on $N$: assume that the theorem holds for $k<N$, so that for each such $k$, we have $\varphi_k=z_1^{M_k}r_k(z_1z_2)$ or $\varphi_k=z_2^{M_k}r_k(z_1z_2)$. By the Gram-Schmidt process,
				\begin{equation}\label{eq:GramSchmidt}
					\varphi_N=z_1^Az_2^B-\sum_{k=0}^{N-1} \frac{\ip{z_1^Az_2^B,\varphi_k}_f}{\norm{\varphi_k}^2_f}\varphi_k.
				\end{equation}
			By Lemma \ref{lem:ogstar}, the inner products in \eqref{eq:GramSchmidt} are zero except when $\varphi_k$ contains a monomial of the form $z_1^{A-J}z_2^{B-J}$ for some $J\leq \min\set{A,B}$. This can be rewritten 
			\begin{equation} \label{eq:goodform}
				z_1^{A-J}z_2^{B-J}=z_1^{B+M-J}z_2^{B-J}.
			\end{equation}
			Any $\varphi_k$ that contains a term of the form \eqref{eq:goodform} contains only monomials that can be written as $z_1^M(z_1z_2)^j$ (by the inductive hypothesis). Therefore, every term of $\varphi_N$ can be written as $z_1^M(z_1z_2)^j$, so $\varphi_N=z_1^M r(z_1z_2)$.
		\end{proof}
		Thus, determining two-variable orthogonal polynomials reduces to finding one variable-polynomials, one family for each row in Figure \ref{fig:triangle}. It is also apparent that all off-diagonal orthogonal polynomials vanish at the origin, confirming what we had already seen from forming successive differences of the corresponding optimal approximants.

		In the particular case $\mathcal{H}=H^2$ with weight $f(z_1,z_2)=1-z_1z_2$, we obtain orthogonal polynomials of a particularly attractive form: here, the $r_N(x)$ can be shown to be the orthogonal polynomials in the one variable weight $1-x$, as in \cite[p.86]{Simon1}.
			
		\begin{corollary}
		For $n=0,1,\ldots$, let
		\[r_n(x)=\frac{1}{n+1}\sum_{k=0}^n(k+1)z^k.\] 
		Then the polynomials
		\[\varphi^{(1)}_{M,m}(z_1,z_2)=z_1^Mr_m(z_1z_2) \quad \text{and}\quad \varphi^{(2)}_{N,n}(z_1,z_2)=z_2^Nr_n(z_1z_2), \]
		with $M,m,N,n\in \N_0$, form an orthogonal basis for $H^2_{1-z_1z_2}(\D^2)$.
		\end{corollary}
		\begin{proof}
		It suffices to note that multiplication by $z_1$ and by $z_2$ is an isometry on $H^2(\D^2)$, meaning that the orthogonality conditions along each row of Figure \ref{fig:triangle} reduce to a condition for the main diagonal, where  orthogonal polynomials can be recovered from the optimal approximants to $1/(1-z_1z_2)$.
		\end{proof}

\section{Zero sets and the Shanks Conjecture}\label{sShanks}
	\subsection{Zero sets and the Shanks conjecture}\label{sZerosShanks}
		We turn to a discussion of zero sets of optimal approximants in several variables. It is natural to ask whether optimal approximants in $\mathcal{H}(\Omega)$ are zero-free in $\Omega$. A variation of this question is whether the assumption that $f(z)\neq 0$ for $z\in \Omega$ implies that the optimal approximants to $1/f$ inherit the zero-free property.

		The classical theory of orthogonal polynomials for $L^2$ can be used to show that optimal approximants in $H^2(\D)$ are zero-free on the closed unit disk $\overline{\D}$ for an arbitrary target function $f$: this problem was addressed by Chui in \cite{Chui80}. In \cite{JLMS16}, an analogous result was established for Dirichlet-type spaces $D_{\alpha}$ for $\alpha\geq 0$: 
		 if $f\in D_{\alpha}$, $f(0)\neq 0$, then $\spn(z)\neq 0$ for all $z\in \overline{\D}$. By contrast, when $\alpha<0$, there are functions $f\in D_{\alpha}$ whose optimal approximants vanish  inside $\mathbb{D}$; in fact, this can happen even for cyclic $f$, which in particular means that $f(z)\neq 0$ in $\D$. However, the zero sets $\mathcal{Z}(\spn)$ always omit a disk $D(0,r(\alpha))$ whose radius is strictly smaller than $1$: it was shown in \cite{JLMS16} that this statement holds with $r(\alpha)=2^{\alpha/2}$. This was sharpened in the subsequent paper \cite{JentZeros19}, and a sharp estimate on $f(\alpha)$ was given for the Hardy space $H^2(\D)$ and the Bergman space $A^2(\D)$.
		 
		As was explained in Section \ref{filtersub}, non-vanishing of optimal approximants has ramifications for filter design, and zero set problems for optimal approximants in $H^2(\D^2)$ have been investigated since the early 70s. In their 1972 paper \cite{Shanks72},  Shanks, Treitel, and Justice conjectured that optimal approximants to $1/f$ for any polynomial $f$ would be zero-free in the bidisk: in subsequent papers in the engineering community, this became known as the {\it Shanks conjecture}.

		A few years later, this strong version of the Shanks conjecture was disproved. In \cite{GKcounter75}, Genin and Kamp exhibited a counterexample, and in \cite{GKstab77}, a method to construct polynomials yielding optimal approximants with zeros in the bidisk was presented. For completeness, we present a simplified version of their counterexample.
		\begin{ex}[Genin-Kamp, 1975]
			Let 
			\begin{multline}
				f(z_1,z_2)=1 -  z_{1} -  z_{2} -  z_{1}^{2} + 4 z_{1} z_{2} -  z_{2}^{2} + 2 z_{1}^{3} - 2 z_{1}^{2} z_{2} - 2 z_{1} z_{2}^{2} \\+ 2 z_{2}^{3} -  z_{1}^{3} z_{2} + 4 z_{1}^{2} z_{2}^{2} -  z_{1} z_{2}^{3} -  z_{1}^{3} z_{2}^{2} -  z_{1}^{2} z_{2}^{3}.
			\end{multline}
			
			For this polynomial, we have the optimal approximant
			\[p_2^*(z_1,z_2)=\frac{39}{1165} + \frac{23}{1165} z_{1} + \frac{23}{1165} z_{2}\]
			which vanishes in the bidisk, for instance at $(z_1,z_2)=\left(\frac{9}{10}e^{3i},\,-\frac{9}{10}e^{3i} - \frac{39}{23}\right)$. Note that the original function $f$ also has zeros in the bidisk.
		\end{ex}

		After the full Shanks conjecture had been disproved, efforts were made to prove a weaker versions of the Shanks conjecture where non-vanishing of optimal approximants in the bidisk is supposed to follow from additional assumptions on $f$. For instance, Delsarte, Genin, and Kamp state a  ``weakest form of Shanks' conjecture'' in \cite{DGKasym80} where non-vanishing of the target polynomial $f$ on the closed bidisk $\overline{\D^2}$ would guarantee that the optimal approximants to $1/f$ are zero-free in $\D^2$. An intermediate version might be to ask that the polynomial $f$ be cyclic in $H^2(\D^2)$ in order to ensure that the optimal approximants $\spn$ have no zeros in $\D^2$; this, as shown in \cite{NGN70}, is equivalent to asking that $f$ itself have no zeros in the {\it open} bidisk.	

		The paper \cite{RedRedSwa84} claimed to establish the weak Shanks conjecture of \cite{DGKasym80}, but in \cite{DGKdisproof85}, Delsarte, Genin, and Kamp show that this purported proof fails.  As far as the authors are aware, the weak Shanks conjecture remains open for the Hardy space of the bidisk:
		\begin{conjecture}[Weakest form of the Shanks conjecture]
			Suppose $f\in \C[z_1,z_2]$ satisfies $f(z)\neq 0$ for $z\in \overline{\D^2}$. Then the $H^2(\D^2)$-optimal approximants to $1/f$ are zero-free in $\D^2$.
		\end{conjecture}
		We have not been able to settle the Shanks conjecture in its weakest form in $H^2$. However, we now demonstrate that it fails in other function spaces of the bidisk, including the Bergman space $A^2(\D^2)$.

	 	\begin{ex}[Counterexample to the Shanks conjecture for the Bergman space]\label{exBergmanCounter}
			Consider the irreducible polynomial
			\begin{equation} \label{eq:a2Zeros}
			 	b(z_1,z_2)=-4+3z_1-z_1^2+3z_2-2z_1z_2+z_1^2z_2-z_2^2+z_1z_2^2.
			\end{equation}
	    This polynomial is the denominator of a rational inner function in the bidisk constructed in \cite{BPSprep}, and hence it follows that $b$ has no zeros in the bidisk and, in particular, is a cyclic vector in the Bergman space $A^2(\D)$, viz. \cite{TAMS16}. However, $b$ does have a single boundary zero at $(1,1)\in \mathbb{T}^2$.

			The second non-constant optimal approximant to $1/b$ can be computed,
			\begin{equation}\label{eq:A2counter}
				p_{ 2 } ^*= \frac{4}{835} \left(-\frac{1267}{27} - 24z_{1} -  24z_{2}\right),	
			\end{equation}
			and has zeros inside the bidisk.
			\begin{figure}[h!]
				\centering
				\includegraphics[width=.5\textwidth]{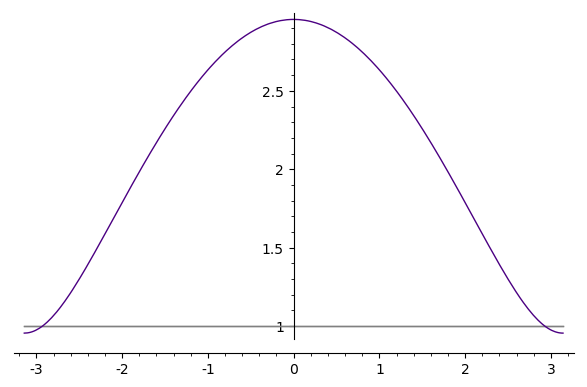}
				\caption{\small Solving \eqref{eq:A2counter} $p_2^*(z_1,e^{it})=0$ for $z_1$ and plotting $\abs{z_1}$ against $t\in(0,2\pi)$. Note that $p_2^*$ is symmetric in $z_1$ and $z_2$.}
			\end{figure}

			We now dilate $b$ to $$\tilde{b}(z_1,z_2)=b\left(\frac{99}{100}z_1,\frac{99}{100}z_2\right).$$ 
			Note that (as can be seen in Figure \ref{fig:dzeros}) the zeros of $\tilde{b}$ are now strictly outside the closed bidisk.

			\begin{figure}[h!]
			    \centering
			    \begin{subfigure}[b]{0.4\textwidth}
			        \includegraphics[width=\textwidth]{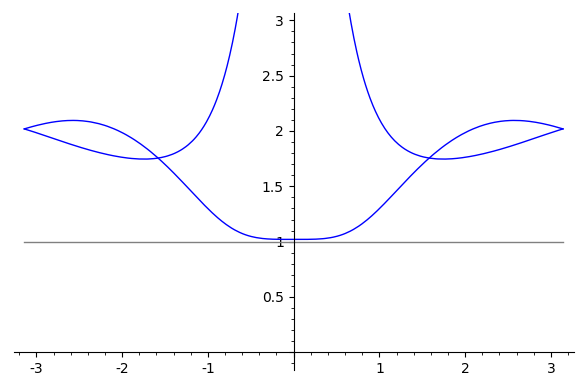}
			        \label{fig:dzerosz1}
			        \caption{\small Solving $\tilde{b}(z_1,e^{it})=0$ for $z_1$ and plotting $\abs{z_1}$ against $t\in(0,2\pi)$}
			    \end{subfigure}
			    \qquad
			    \begin{subfigure}[b]{0.4\textwidth}
			        \includegraphics[width=\textwidth]{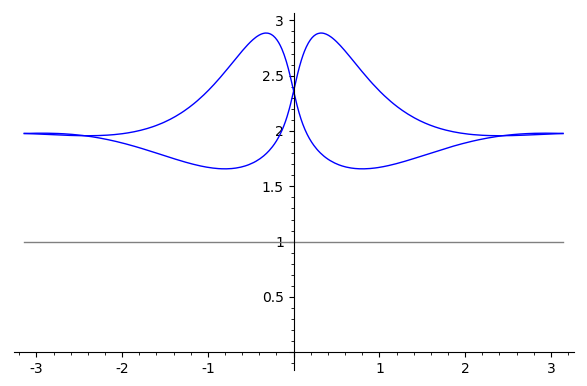}
			        \label{fig:dzerosz2}
			        \caption{\small Solving $\tilde{b}(e^{it},z_2)=0$ for $z_2$ and plotting $\abs{z_2}$ against $t\in(0,2\pi)$.}
			    \end{subfigure}
			    \caption{\small Facial zero sets of $\tilde{b}$}\label{fig:dzeros}
			\end{figure}

			\begin{figure}[h!]				
				\centering
				\includegraphics[width=.5\textwidth]{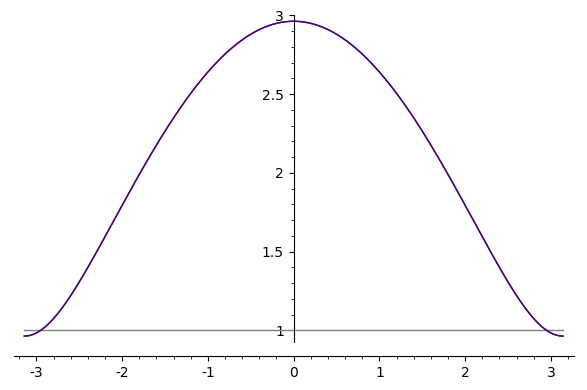}
				\caption{\small Solving $p_2^*(z_1,e^{it})=0$ for $z_1$ and plotting $\abs{z_1}$ against $t\in(0,2\pi)$. Note that $p_2^*$ is symmetric in $z_1$ and $z_2$.}\label{fig:dOAzeros}
			\end{figure}

			However, the optimal approximant
			\begin{equation}
				p_2^*=\frac{11316790431936000000000000}{98483870117907418000870963}
							\brkt{-\frac{5554782671089}{2829197607984} -  z_{1} -  z_{2}	}
			\end{equation}
			has zeros inside the closed bidisk, as seen in Figure \ref{fig:dOAzeros}.				
		\end{ex}
		\begin{rmks}
			The same function $b$ produces optimal approximants which vanish in the bidisk for $\alpha_1=\alpha_2=-.85$. Similarly, choosing $\alpha_1=0,\, \alpha_2=-3$, $b$ also yields zeros in the bidisk for $p^{\ast}_2$. 
		\end{rmks}

	\subsection{Reproducing kernel methods}
		One faces several difficulties when seeking to extend the results of \cite{Chui80} and \cite{JLMS16,  JentZeros19} on the location of zero set of optimal approximants to function spaces in several variables. Zeros of a polynomial, or indeed any holomorphic function of several complex variables, are never isolated, and we no longer have access to the fundamental theorem of algebra. We briefly revisit the reproducing kernel arguments in \cite{JLMS16} in the multi-variable setting to see how these facts block a straight-forward extension of the proof.

		Let  $\spn$ be an optimal approximant to $1/f$ in $\cd_{\alpha_1,\alpha_2}$. Then, as is explained in Section \ref{ss:RKHS}, we have
			\[\mathfrak{K}_n(z,0)=\spn(z)f(z),\]
		where $\mathfrak{K}_n(\cdot,0)$ is the reproducing kernel at $0$ for $f\cdot \Pn$. Suppose for a moment that $\spn$ is of the form
			\[\spn(z_1,z_2)=(P(z_1,z_2)-w_0)Q(z_1,z_2)\]
		for some $w_0\in \C$, some $P\in \C[z_1,z_2]$ vanishing at the origin, and some $Q \in \C[z_1,z_2]$. We seek to determine some set $K\subset \Omega$ such that $w_0-P(z)\neq 0$ for $z\in K$. 
As in \cite[Section 4]{JLMS16}, we can write
			\[w_0Q(z)f(z)=P(z)Q(z)f(z)-\mathfrak{K}_n(z,0)\]
		and since $PQf$ vanishes at the origin and is an element of $f\cdot \Pn$, we get $PQf\perp \mathfrak{K}_n(\cdot,0)$ by appealing to the reproducing property of $\mathfrak{K}_n(z,0)$. This in turn implies that
			\[|w_0|^2\|Qf\|^2_{\mathcal{H}}=\|PQf\|^2_ {\mathcal{H}}+\|\mathfrak{K}_n(\cdot,0)\|_{\mathcal{H}}^2.\]
		Since $\|\mathfrak{K}_n(\cdot, 0)\|\geq 0$, it follows that
			\begin{equation}
|w_0|^2\|Qf\|^2-\|PQf\|^2\geq 0.
\label{eq:zeroout}
\end{equation}
		Up to this point, the argument is identical to that in \cite{JLMS16}. Now, in one variable, the assumption that $w_0\in\C$ is a zero of $\spn$ allows us to take $P(z)=z$. In many function spaces of interest, such as the Dirichlet spaces, one has $\|zf\|\geq C(\mathcal{H})\|f\|$ for some easily computable constant $C(\mathcal{H})$, and this allows us to conclude that from \eqref{eq:zeroout} that $|w_0|^2-C(\mathcal{H})^2\geq 0$, thus obtaining a lower bound on the location of zeros of the one-variable polynomial $\spn(z)$.

		In several variables, there is no distinguished form of $P$ and even if we restrict ourselves to some prescribed factor $P$, we are left with the task of estimating the ratio $\|PQf\|/\|Qf\|$ from below, and this does not seem like an easy task. Finally, assuming a lower bound on $|w_0|$ is obtained in this way, we would in addition need to analyze whether this lower bound places $w_0$ outside the range of $P(z)$ on some subset $\D^2$.

		We do obtain the following, again by leveraging one-variable arguments.
		\begin{lemma}
		Let $\spn$ be an optimal approximant to $1/f$ in $\cd_{\alpha_1,\alpha_2}$ and suppose $\spn(z)=(w_0-z_1z_2)Q(z)$ for some $Q\in \C[z_1,z_2]$ with $Q(z)\neq 0$ for $z\in \D^2$. 

		If $\alpha_1\geq 0$ and $\alpha_2\geq 0$, then $\spn$ does not vanish in the bidisk. If $\alpha_1<0$ and $\alpha_2<0$, then $\spn$ does not vanish in $D(0,2^{(\alpha_1+\alpha_2)/2})\times D(0,2^{(\alpha_1+\alpha_2)/2})$. 
		\end{lemma}
		One can imagine variations of the above argument for other special factors such as $P(z_1,z_2)=z_1$, but it would clearly be desirable to find a general methods for analyzing zero sets of optimal approximants in several variables.

		\begin{qst}
		Let $\{\spn\}$ be optimal approximants to $f\in\da\setminus\{0\}$. Is there a compact set $K\subset \D^2$ such that $\spn(z)\neq 0$ for $z\in K$ and all $n$? 

Similarly, if $\{\spn\}$ are optimal approximants to $f\in H^2_d$, is there a compact set $K\subset \B^d$ such that $\spn(z)\neq0$ for $z\in K$ and all $n$? 
		\end{qst}

\section{Explicit computations for  $f=1-a(z_1+z_2)$}\label{sComp}	
In this section, we record some observations concerning optimal approximants and orthogonal polynomials associated with a polynomial that vanishes at a single boundary point. More precisely, we consider $f=1-a(z_1+z_2)$ which can be viewed as a natural analog of the classical one-variable weight $1-z$. In the case of the Drury-Arveson space $H^2_2$ in 
$\mathbb{B}^2$, we take $a=\frac{1}{\sqrt{2}}$, and are able to exhibit closed formulas for some of the optimal approximants. Then, we turn to the bidisk, set $a=1$, compute some low-degree optimal approximants and orthogonal polynomials for Dirichlet-type spaces, and note that the situation is more complicated. This gives an example where the ball and bidisk theories are different. In Section \ref{sRKHS}, we were able to use a diagonal embedding to handle both $\mathbb{B}^2$ and $\mathbb{D}^2$ but here, we exploit the fact that the ball, unlike the bidisk, is invariant under unitary transformations. 

Throughout, we use degree lexicographical ordering, as in Section \ref{ssOA_def}.

	\subsection{Optimal approximants and orthogonal polynomials for $H^2_2$}
Consider $f(z_1,z_2)=1-\frac{1}{\sqrt{2}}(z_1+z_2)$, which vanishes at $(\frac{1}{\sqrt{2}},\frac{1}{\sqrt{2}})$ in the unit sphere $\mathbb{S}^2$. 
Using the Grammian method described in Section \ref{ssGrammian}, we compute the first optimal approximants for $1/f$:
		\begin{align*}
			p^*_{ 0 } &=  \frac{1}{2}\\
			p^*_{ 1 }&=  \frac{1}{12} \left(7+ 2 \sqrt{2} z_{1}\right)\\
			p^*_{ 2 }&=  \frac{1}{6} \left(4+  \sqrt{2} z_{1} + \sqrt{2} z_{2}\right)\\
			p^*_{ 3 }&=  \frac{1}{48} \left(33+ 10\sqrt{2} z_{1} + 8\sqrt{2} z_{2} + 6 z_{1}^{2}\right)\\
			p^*_{ 4 }&=  \frac{1}{48} \left(35+ 12\sqrt{2} z_{1} + 10\sqrt{2} z_{2} + 6z_{1}^{2} + 12z_{1} z_{2}\right)\\
			p^*_{ 5 }&=  \frac{1}{8} \left(6+ 2\sqrt{2} z_{1} + 2\sqrt{2} z_{2} + z_{1}^{2} + 2z_{1} z_{2} + z_{2}^{2}\right).
		\end{align*}
From these, we can compute orthogonal polynomials in the weighted space as discussed in Section \ref{ssOAsAndOGs}:
		\begin{align*}
			\phi_{ 0 }&=  1 \\
			\phi_{ 1 }&=  \frac{1}{12}\left(1 + 2\sqrt{2} z_{1}\right) \\
			\phi_{ 2 }&=  \frac{1}{12}\left(1 + 2\sqrt{2} z_{2}\right) \\
			\phi_{ 3 }&=  \frac{1}{48}\left(1 + 2\sqrt{2} z_{1} + 3z_{1}^{2} \right)\\
			\phi_{ 4 }&=  \frac{1}{24}\left(1 + \sqrt{2} z_{1} + \sqrt{2} z_{2} + 6 z_{1} z_{2}\right) \\
			\phi_{ 5 }&=  \frac{1}{48}\left(1 + 2\sqrt{2} z_{2} + 3z_{2}^{2} \right)\\
			\phi_{ 6 }&=  \frac{1}{160}\left(1 + 2\sqrt{2} z_{1} + 6 z_{1}^{2} + 8 \sqrt{2} z_{1}^{3}\right) \\
			\phi_{ 7 }&=  \frac{3}{160} + \frac{1}{40} \sqrt{2} z_{1} + \frac{1}{80} \sqrt{2} z_{2} + \frac{3}{80} z_{1}^{2} + \frac{3}{40} z_{1} z_{2} + \frac{3}{20} \sqrt{2} z_{1}^{2} z_{2} \\
			\phi_{ 8 }&=  \frac{3}{160} + \frac{1}{80} \sqrt{2} z_{1} + \frac{1}{40} \sqrt{2} z_{2} + \frac{3}{40} z_{1} z_{2} + \frac{3}{80} z_{2}^{2} + \frac{3}{20} \sqrt{2} z_{1} z_{2}^{2} \\
			\phi_{ 9 }&=  \frac{1}{160} + \frac{1}{80} \sqrt{2} z_{2} + \frac{3}{80} z_{2}^{2} + \frac{1}{20} \sqrt{2} z_{2}^{3}. 
		\end{align*}
The appearances of $p_2^*$ and $p_5^*$ are easy to explain. 
\begin{proposition}\label{unitarytrick}
Let $N\in \mathbb{N}$ be such that $\mathcal{P}_N$ contains all two-variable monomials of total degree $n$, and no monomials of total degree $n+1$. 

Then the $N$th optimal approximant to $1/(1-\frac{1}{\sqrt{2}}(z_1+z_2))$ is given by
\[p_N^*(z_1,z_2)=r_N\left(\frac{z_1+z_2}{\sqrt{2}}\right),\]
where $r_n(x)=\frac{1}{n+1}\sum_{k=0}^n(k+1)z^k$.
\end{proposition}
\begin{proof}
Let 
\[U=\frac{1}{\sqrt{2}}\left(\begin{array}{cc}1 & 1\\-1 & 1\end{array}\right)\in U_2(\mathbb{C})\]
act on $\mathbb{C}^2$ by left multiplication and note that
\[f=F\circ U,\]
where $F=1-z_1$. Now let $p\in \mathcal{P}_N$. Since $p^*_N$ defined above is in $\mathcal{P}_N$, we can write $p=p^*_N+(p-p^*_N)=Q_1+Q_2$, and using the invariance of the $H^2_2$-norm under unitaries, we obtain
\[\|p f-1\|_{H^2_2}=\|Q_1f-1+Q_2f\|_{H^2_2}=\|r_N F-1+(Q_2\circ U^{-1})\cdot F\|_{H^2_2};\]
note that $r_n$ and $F$ are one-variable functions. Since monomials are orthogonal in $H^2_2$, we obtain a lower bound by stripping out contributions that do not only depend on $z_1$:
\[\|pf-1\|_{H^2_2}\geq \|r_NF-1+(Q_2\circ U^{-1})(\cdot, 0)\cdot F\|_{H^2_2}.\]
Since $r_N$ is the $N$th order optimal approximant to $1/F$ in the Hardy space $H^2(\D)$, and since the $H^2_2$-norm restricted to functions of $z_1$ only reduces to the one-variable Hardy norm, the norm on the right is bounded below by $\|r_NF-1\|_{H^2}$. In the above argument, we have equality throughout provided $Q_2=0$, and the result now follows.
\end{proof}
As a corollary, we get from the one-variable results in \cite{JAM15,FMS14} that $f=1-\frac{1}{\sqrt{2}}(z_1+z_2)$ is cyclic in the Drury-Arveson space, with distance estimate
\[\nu_N(1-(z_1+z_2)/\sqrt{2}, H^2_2)\asymp \frac{1}{N+1}.\]
By contrast, in Example \ref{ex:DAdiag}, we noted that
\[\nu_{\oslash N}(1-\sqrt{2}z_1z_2, H^2_2)\asymp \frac{1}{\sqrt{N+1}}.\]
This seems to suggest that having a bigger boundary zero set may increase the optimal distance of a polynomial (cf. the discussion in \cite[Section 5]{PJM15}).

We leave a full determination of $H^2_2$-optimal approximants and orthogonal polynomials to future work.
	\subsection{Optimal approximants and orthogonal polynomials for $\da$}\label{ssCompBidisk}
We now turn to the bidisk and $f(z_1,z_2)=2-z_1-z_2$, and present some optimal approximants for $1/f$ in three Dirichlet-type spaces. 
		\begin{ex}[The Hardy Space, $\alpha=0$]
	Again using Grammians, we begin by computing some optimal approximants. Interestingly, while the $z_1^5$ coefficient in $p^*_{20}$ is negative, the $z_1^5$ coefficient is positive when that term first appears in $p^*_{15}$, and is first negative in $p^*_{17}$. 
			\begin{align*}
				p^*_{ 0 } &=  \frac{1}{3}\\
				p^*_{ 1 } &=  \frac{1}{8}\left(3 +  z_{1}\right)\\
				p^*_{ 2 } &=  \frac{1}{17} \left( 7+ 2z_{1} + 2z_{2}\right)\\
				p^*_{ 3 } &=  \frac{1}{223}\left( 93+ 30 z_{1} + 26 z_{2} + 10 z_{1}^{2}\right)\\
				p^*_{ 4 } &=  \frac{1}{2039} \left(897+ 342 z_{1} + 310 z_{2} + 80 z_{1}^{2} + 204 z_{1} z_{2}\right)\\
				p^*_{ 5 } &=  \frac{1}{205} \left(91+ 34 z_{1} + 34 z_{2} + 8 z_{1}^{2} + 20 z_{1} z_{2} + 8 z_{2}^{2}\right)\\
				\vdots &~		\\
				p^*_{ 20 }&=  0.4767094 + 0.2150641 z_{1} + 0.2150641 z_{2} + 0.08684609 z_{1}^{2} \\
				&\qquad+ 0.1891688 z_{1} z_{2} + 0.08684609 z_{2}^{2} + 0.02794644 z_{1}^{3} + 0.1121122 z_{1}^{2} z_{2} \\
				&\qquad+ 0.1121122 z_{1} z_{2}^{2} + 0.02794644 z_{2}^{3} + 0.005193106 z_{1}^{4} + 0.04785621 z_{1}^{3} z_{2}\\
				&\qquad + 0.08249469 z_{1}^{2} z_{2}^{2} + 0.04785621 z_{1} z_{2}^{3} + 0.005193106 z_{2}^{4} \\
				&\qquad+ \left(-0.0002349534\right) z_{1}^{5} + 0.01179593 z_{1}^{4} z_{2} + 0.03555798 z_{1}^{3} z_{2}^{2} \\
				&\qquad+ 0.03555798 z_{1}^{2} z_{2}^{3} + 0.01179593 z_{1} z_{2}^{4} + \left(-0.0002349534\right) z_{2}^{5}
			\end{align*}
			We next compute orthogonal polynomials in the corresponding weighted space, again as discussed in Section \ref{ssOAsAndOGs}.
			\begin{align*}
				\phi_{ 0 }&=  1 \\
				\phi_{ 1 }&=  \frac{1}{24}\left(1 + 3 z_{1}\right) \\
				\phi_{ 2 }&=  \frac{1}{136}\left(5 -  z_{1} + 16 z_{2}\right) \\
				\phi_{ 3 }&=  \frac{2}{3791}\left(10 + 32z_{1} - 2 z_{2} + 85 z_{1}^{2}\right) \\
				\phi_{ 4 }&=  \frac{6}{454697}\left(1734 + 2516 z_{1} + 2686 z_{2} - 425 z_{1}^{2} + 7582 z_{1} z_{2}\right) \\
				\phi_{ 5 }&=  \frac{4}{417995}\left(416 - 187z_{1} + 1444z_{2} - 22z_{1}^{2} - 260 z_{1} z_{2} + 4078z_{2}^{2} \right)
			\end{align*}
		\end{ex}
		\begin{ex}[The Dirichlet Space, $\alpha=1$]
			Again, we compute optimal approximants to $1/f$ where $f(z_1,z_2)=2-z_1-z_2$. The negative coefficients appear sooner here, with the $z_1^3$ term being negative first in $p^*_8$, although it is positive when it first appears in $p^*_6.$
			\begin{align*}
				p^*_{ 0 } &=  \frac{1}{4}\\
				p^*_{ 1 } &=  \frac{1}{52}\left(15 + 4 z_{1}\right)\\
				p^*_{ 2 } &=  \frac{1}{60}\left(19 + 4 z_{1} + 4 z_{2}\right)\\
				p^*_{ 3 } &=  \frac{1}{6324} \left(2029+ 484 z_{1} + 412 z_{2} + 132 z_{1}^{2}\right)\\
				p^*_{ 4 } &=  \frac{1}{60260}\left(20941 + 6092 z_{1} + 5660 z_{2} + 792 z_{1}^{2} + 3188 z_{1} z_{2}\right)\\
				p^*_{ 5 } &=  \frac{1}{1372} \left(479+ 136 z_{1} + 136 z_{2} + 18 z_{1}^{2} + 70z_{1} z_{2} + 18 z_{2}^{2}\right)\\
				\vdots &~		\\
				p^*_{ 9 }&=  0.368042 + 0.118042 z_{1} + 0.118042 z_{2} + 0.0242648 z_{1}^{2} + 0.0781293 z_{1} z_{2} \\
				&\qquad+ 0.0242648 z_{2}^{2} + \left(-0.0000894141\right) z_{1}^{3} + 0.0245889 z_{1}^{2} z_{2} \\
				&\qquad+ 0.0245889 z_{1} z_{2}^{2} + \left(-0.0000894141\right) z_{2}^{3}
			\end{align*}
			The first few corresponding orthogonal polynomials are
			\begin{align*}
				\phi_{ 0 }&=  1 \\
				\phi_{ 1 }&=  \frac{1}{26}\left(1 + 2 z_{1}\right) \\
				\phi_{ 2 }&=  \frac{1}{390}\left(11 - 4 z_{1} + 26 z_{2}\right) \\
				\phi_{ 3 }&=  \frac{1}{2635}\left(11 + 26z_{1} - 4 z_{2} + 55 z_{1}^{2} \right)\\
				\phi_{ 4 }&=  \frac{1}{23817765}\left(635209 + 584998z_{1} + 685420 z_{2} - 184107 z_{1}^{2} + 1260057 z_{1} z_{2} \right)\\
				\phi_{ 5 }&=  \frac{1}{10334590}\left(16686 - 20358 z_{1} + 53730 z_{2} - 243 z_{1}^{2} - 19467 z_{1} z_{2} + 135585 z_{2}^{2}\right).
			\end{align*}
		\end{ex}
		\begin{ex}[The Bergman Space, $\alpha=-1$]
			For $f(z_1,z_2)=2-z_1-z_2$, in the Bergman space, the first occurrence of a negative coefficient is for the $z_1^9$ coefficient in $p^*_{47}$, although similarly to the previous cases, the $z_1^9$ coefficient is positive when it first appears in $p^*_{45}$. In $p^*_{54}$ where the $z_2^9$ term first appears, its coefficient is also negative. We only present the first few optimal approximants for $1/f$, as the 47th or 54th polynomials are prohibitively long.
			\begin{align*}
				p^*_{ 0 } &=  \frac{2}{5}\\
				p^*_{ 1 } &=  \frac{1}{143}\left(62 + 24 z_{1}\right)\\
				p^*_{ 2 } &=  \frac{1}{73} \left(34+ 12 z_{1} + 12 z_{2}\right)\\
				p^*_{ 3 } &=  \frac{1}{9587}\left(4502 + 1764 z_{1} + 1572 z_{2} + 672 z_{1}^{2}\right)\\
				p^*_{ 4 } &=  \frac{1}{16211} \left(7802 +3450 z_{1} + 3138 z_{2} + 1092 z_{1}^{2} + 2334 z_{1} z_{2}\right)\\
				p^*_{ 5 } &=  \frac{1}{1547}\left(750 + 328 z_{1} + 328 z_{2} + 104 z_{1}^{2} + 220z_{1} z_{2} + 104 z_{2}^{2}\right)
			\end{align*}
			From these, we can find the orthogonal polynomials.
			\begin{align*}
				\phi_{ 0 }&=  1 \\
				\phi_{ 1 }&=  \frac{24}{715}\left(1 + 5z_{1}\right) \\
				\phi_{ 2 }&=  \frac{12}{10439}\left(28 - 3 z_{1} + 143 z_{2}\right) \\
				\phi_{ 3 }&=  \frac{2688}{699851}\left( + \frac{13728}{699851} z_{1} - \frac{288}{699851} z_{2} + \frac{672}{9587} z_{1}^{2}\right) \\
				\phi_{ 4 }&=  \frac{6}{155414857}\left(302642 + 746491z_{1} + 766719z_{2} - 70798 z_{1}^{2} + 3729343 z_{1} z_{2}\right) \\
				\phi_{ 5 }&=  \frac{2}{148393}\left(262 - 59 z_{1} + 1369 z_{2} - 10 z_{1}^{2} - 131 z_{1} z_{2} + 4988 z_{2}^{2}\right)
			\end{align*}
		\end{ex}
		\begin{rmks}[Negative coefficients]
			Even for this simple target function, for higher order approximants, some coefficients are negative. This is in contrast to the one variable situation with $f(z)=(1-z)^{a}$, $a\geq 0$ real, where coefficients of the optimal approximants can be found as positively weighted sums of the Taylor coefficients of $1/f$. (Explicit computation of optimal approximants for $f(z)=(1-z)^a$ can be found in \cite{JAM15,JentZeros19}.)

			As observed above, there appears to be a relationship between the value of $\alpha$ and the first $p^*_n$ in which negative coefficients appear; roughly that more Dirichlet-like spaces ($\alpha>0$) have negative coefficients appearing sooner and that more Bergman-like spaces ($\alpha<0$) have them occurring later, although we have not carefully examined this.
		\end{rmks}
We hope to further explore optimal approximants and orthogonal polynomials in several variables, including the examples in this section, in a systematic way in future work. 
		
\section*{Acknowledgments}
We thank Catherine B\'en\'eteau for some very helpful conversations, and especially for drawing our attention to several important papers including \cite{Chui80,Izu85}. We are grateful to Michael Hartz, Greg Knese, Stefan Richter, and Hugo Woerdeman for useful conversations and correspondence.

\bibliography{bidisk_OG_OA.bib} 

\begin{thebibliography}{10}

\bibitem{AMBook}
{\sc Agler, J., and McCarthy, J.~E.}
\newblock {\em Pick interpolation and {Hilbert} function spaces}.
\newblock American Mathematical Society, 2000.

\bibitem{arvy}
{\sc Arveson, W.}
\newblock Subalgebras of {$C^*$}-algebras {III}: Multivariable operator theory.
\newblock {\em Acta Math. 181\/} (1998), 159--228.

\bibitem{JAM15}
{\sc B\'{e}n\'{e}teau, C., Condori, A.~A., Liaw, C., Seco, D., and Sola, A.~A.}
\newblock Cyclicity in {Dirichlet}-type spaces and extremal polynomials.
\newblock {\em J. Anal. Math. 126\/} (2015), 259--286.

\bibitem{PJM15}
{\sc B\'{e}n\'{e}teau, C., Condori, A.~A., Liaw, C., Seco, D., and Sola, A.~A.}
\newblock Cyclicity in {Dirichlet}-type spaces and extremal polynomials {II}:
  Functions on the bidisk.
\newblock {\em Pacific J. Math. 276}, 1 (2015), 35--58.

\bibitem{CMB18}
{\sc B\'{e}n\'{e}teau, C., Fleeman, M., Khavinson, D., Seco, D., and Sola,
  A.~A.}
\newblock Remarks on inner functions and optimal approximants.
\newblock {\em Canad. Math. Bull. 61}, 4 (2018), 704--716.

\bibitem{Betalprep1}
{\sc B\'{e}n\'{e}teau, C., Ivrii, O., Manolaki, M., and Seco, D.}
\newblock Simultaneous zero-free approximation and universal optimal polynomial
  approximants.
\newblock {\em arXiv, Preprint, 2018\/}.

\bibitem{JentZeros19}
{\sc B\'{e}n\'{e}teau, C., Khavinson, D., Liaw, C., Seco, D., and Simanek, B.}
\newblock Zeros of optimal polynomial approximants: {Jacobi} matrices and
  {Jentzsch}-type theorems.
\newblock {\em Rev. Mat. Iberoam 35}, 2 (2019), 607--642.

\bibitem{JLMS16}
{\sc B\'{e}n\'{e}teau, C., Khavinson, D., Liaw, C., Seco, D., and Sola, A.~A.}
\newblock Orthogonal polynomials, reproducing kernels, and zeros of optimal
  approximants.
\newblock {\em J. London Math. Soc. 94}, 3 (2016), 726--746.

\bibitem{TAMS16}
{\sc B\'{e}n\'{e}teau, C., Knese, G., Kosi\'{n}ski, L., Liaw, C., Seco, D., and
  Sola, A.}
\newblock Cyclic polynomials in two variables.
\newblock {\em Trans. Amer. Math. Soc. 368}, 12 (2016), 87378--754.

\bibitem{Betalprep2}
{\sc B\'{e}n\'{e}teau, C., Manolaki, M., and Seco, D.}
\newblock Boundary behavior of optimal polynomial approximants.
\newblock {\em arXiv, Preprint, 2019\/}.

\bibitem{BPSprep}
{\sc Bickel, K., Pascoe, J.~E., and Sola, A.}
\newblock Level curve portraits of rational inner functions.
\newblock {\em Ann. Scuola Norm. Sup. di Pisa, Cl. Scienze, to appear\/}.

\bibitem{BScyc84}
{\sc Brown, L., and Shields, A.~L.}
\newblock Cyclic vectors in the {Dirichlet} space.
\newblock {\em Trans. Amer. Math. Soc 285\/} (1984), 269--304.

\bibitem{ChMaRo19}
{\sc Cheng, R., Mashreghi, J., and Ross, W.~T.}
\newblock Inner functions in reproducing kernel spaces.
\newblock In {\em Analysis of operators on function spaces}, A.~Aleman,
  H.~Hedenmalm, D.~Khavinson, and M.~Putinar, Eds. Birkhauser, 2019,
  pp.~167--211.

\bibitem{Chui80}
{\sc Chui, C.~K.}
\newblock Approximation by double least-squares inverses.
\newblock {\em J. Math. Anal. Appl. 75\/} (1980), 149--163.

\bibitem{DGKalg79}
{\sc Delsarte, P., Genin, Y.~V., and Kamp, Y.~G.}
\newblock Planar least squares inverse polynomials: part {I} - algebraic
  properties.
\newblock {\em IEEE Trans. Circuits and Systems CAS-26}, 1 (1979), 59--66.

\bibitem{DGKasym80}
{\sc Delsarte, P., Genin, Y.~V., and Kamp, Y.~G.}
\newblock Planar least squares inverse polynomials: part {II} - asymptotic
  behavior.
\newblock {\em SIAM J. Alg. Disc. Meth. 1}, 3 (1980), 336--344.

\bibitem{DGKdisproof85}
{\sc Delsarte, P., Genin, Y.~V., and Kamp, Y.~G.}
\newblock Comments on ``proof of a modified form of {Shanks'} conjecture on the
  stability of 2-{D} planar least square inverse polynomials and its
  implications''.
\newblock {\em IEEE Trans. on Circuits and Systems CAS-32}, 9 (1985), 966--968.

\bibitem{DurBook}
{\sc Duren, P.~L.}
\newblock {\em Theory of {$H^p$} spaces}.
\newblock Dover Publications Inc., 2000.

\bibitem{ElFBook}
{\sc El-Fallah, O., Kellay, K., Mashreghi, J., and Ransford, T.}
\newblock {\em A primer on the {Dirichlet} space}.
\newblock Cambridge University Press, 2015.

\bibitem{FMS14}
{\sc Fricain, E., Mashreghi, J., and Seco, D.}
\newblock Cyclicity in reproducing kernel {Hilbert} spaces of analytic
  functions.
\newblock {\em Comput. Methods Funct. Theory 14}, 4 (2014), 665--680.

\bibitem{GKcounter75}
{\sc Genin, Y.~V., and Kamp, Y.~G.}
\newblock Counterexample in the least-squares inverse stabilisation of {2D}
  recursive filters.
\newblock {\em Elec. Letters 11\/} (1975), 330--331.

\bibitem{GKstab77}
{\sc Genin, Y.~V., and Kamp, Y.~G.}
\newblock Two-dimensional stability and orthogonal polynomials on the
  hypercircle.
\newblock {\em Proc. of the IEEE 65\/} (1977), 873--881.

\bibitem{GWsiam07}
{\sc Geronimo, J.~S., and Woerdeman, H.~J.}
\newblock Two variable orthogonal polynomials on the bicircle and structured
  matrices.
\newblock {\em SIAM J. Matrix Anal. Appl. 29\/} (2007), 796--825.

\bibitem{HedBook}
{\sc Hedenmalm, H., Korenblum, B., and Zhu, K.}
\newblock {\em Theory of {Bergman} spaces}.
\newblock Springer-Verlag, 2000.

\bibitem{Izu85}
{\sc Izumino, S.}
\newblock Generalized inverses of {Toeplitz} operators and inverse
  approximation in {$H^2$}.
\newblock {\em Tohoku J. Math. (2) 37\/} (1985), 95--99.

\bibitem{JS73}
{\sc Justice, J., and Shanks, J.}
\newblock Stability criterion for $n$-dimensional digital filters.
\newblock {\em IEEE Trans. Automatic Control AC-18\/} (1973), 284--286.

\bibitem{JAM19}
{\sc Knese, G., Kosi\'{n}ski, L., Ransford, T., and Sola, A.~A.}
\newblock Cyclic polynomials in anisotropic {Dirichlet} spaces.
\newblock {\em J. Anal. Math. 130\/} (2019), 23--47.

\bibitem{Le19}
{\sc Le, T.}
\newblock Inner functions in weighted {Hardy} spaces.
\newblock arXiv: 1912:05715, 2019.

\bibitem{NGN70}
{\sc Neuwirth, J.~H., Ginsberg, J., and Newman, D.~J.}
\newblock Approximation by $\{f(kx)\}$.
\newblock {\em J. Funct. Anal. 5\/} (1970), 194--203.

\bibitem{RedRedSwa84}
{\sc Reddy, P., Reddy, D. R.~R., and Swamy, M.}
\newblock Proof of a modified form of shanks' conjecture on the stability of
  $2$-d planar least square inverse polynomials and its implications.
\newblock {\em IEEE Trans. Circuits and Systems 31\/} (1984), 1009.

\bibitem{RicSun12}
{\sc Richter, S., and Sundberg, C.}
\newblock Cyclic vectors in the {Drury}-{Arveson} space.
\newblock Slides from ESI talk, 2012.

\bibitem{RicSun16}
{\sc Richter, S., and Sunkes, J.}
\newblock Hankel operators, invariant subspaces, and cyclic vectors in the
  {Drury}-{Arveson} space.
\newblock {\em Proc. Amer. Math. Soc. 144\/} (2016), 2575--2586.

\bibitem{Rob63}
{\sc Robinson, E.~A.}
\newblock Structural properties of stationary stochastic processes with
  applications.
\newblock In {\em Time Series Analysis}, M.~Rosenblatt, Ed. Wiley, New York,
  1963.

\bibitem{RudSBook}
{\sc Rudin, W.}
\newblock {\em Function theory in polydisks}.
\newblock W. A. Benjamin, Inc., 1969.

\bibitem{SecoSurvey}
{\sc Seco, D.}
\newblock Some problems on optimal approximants.
\newblock In {\em Recent progress on operator theory and approximation in
  spaces of analytic functions}, C.~B\'{e}n\'{e}teau, A.~A. Condori, C.~Liaw,
  W.~T. Ross, and A.~A. Sola, Eds. Amer. Math. Soc., Providence, RI, 2016,
  pp.~193--205.

\bibitem{ShaBook}
{\sc Shalit, O.~M.}
\newblock Operator theory and function theory in {Drury}-{Arveson} space and
  its quotients.
\newblock In {\em Handbook of Operator Theory}, D.~Alpay, Ed. Springer-Verlag,
  2015, pp.~1125--1180.

\bibitem{Shanks72}
{\sc Shanks, J.~L., Treitel, S., and Justice, J.~H.}
\newblock Stability and synthesis of two-dimensional recursive filters.
\newblock {\em IEEE Trans. on Audio and Electroacoustics AU-20}, 2 (1972),
  115--128.

\bibitem{ShaShi62}
{\sc Shapiro, H.~S., and Shields, A.~L.}
\newblock On the zeros of functions with finite {Dirichlet} integral and some
  related function spaces.
\newblock {\em Math. Z. 80\/} (1962), 217--229.

\bibitem{Simon1}
{\sc Simon, B.}
\newblock {\em Orthogonal Polynomials on the Unit Circle, Part 1: Classical
  Theory}.
\newblock American Mathematical Society, 2005.

\bibitem{S15}
{\sc Sola, A.}
\newblock A note on {Dirichlet}-type spaces and cyclic vectors in the unit ball
  of $\mathbb{C}^2$.
\newblock {\em Arch. Math. (Basel) 104\/} (2015), 247--257.

\end{thebibliography}
\bibliographystyle{acm}

\end{document}